\documentclass[11pt,oneside,a4paper,mathscr]{amsart}\openup1.6\jot
\usepackage{amsthm}
\usepackage{amsmath}
\usepackage{amsrefs} 
\usepackage{amssymb}
\usepackage{amsfonts}
\usepackage{amscd}
\usepackage{mathrsfs} 
\usepackage{mathtools}
\usepackage[dvipdfmx]{graphicx}
\usepackage{bm}   
\usepackage[all]{xy}
\usepackage{ulem}
\usepackage[all]{xy}
\usepackage{ulem}
\newtheorem{theorem}{Theorem}[section]
\newtheorem{lemma}[theorem]{Lemma}
\newtheorem{proposition}[theorem]{Proposition}
\newtheorem{corollary}[theorem]{Corollary}
\newtheorem{claim}[theorem]{Claim}

\theoremstyle{definition}
\newtheorem{definition}[theorem]{Definition}

\newtheorem{remark}[theorem]{Remark} 

\def\C{\mathbb{C}} 
\def\R{\mathbb{R}}

\def\P{\mathbb{P}}
\def\Z{\mathbb{Z}}

\def\i{{\tt{i}}}
\def\del{\partial}
\def\id{\mathrm{id}}

\newcommand{\End}{\mathrm{End}}
\newcommand{\Hom}{\mathrm{Hom}}

\newcommand{\homscr}{{\mathscr{H}\! \! om}}

\def\({\left(}
\def\){\right)}
\def\<{\langle}
\def\>{\rangle}
\newcommand{\simeqto}{\xrightarrow{\sim}}

\newcommand{\pole}[1]{\ensuremath{(\!(#1)\!)} }
\newcommand{\conv}[1]{\ensuremath{(\!\{#1\}\!)}}
\newcommand{\cal}[1]{\mathcal{#1}}
\renewcommand{\scr}[1]{\mathscr{#1}}
\renewcommand{\frak}[1]{\mathfrak{#1}}

\def\GL{{\mathrm{GL}}}

\def\per{{\mathrm{per}}}

\def\mild{{\mathrm{mild}}}
\def\asy{{\mathrm{asy}}}
\def\gr{{\mathrm{gr}}}

\def\RH{{\mathrm{RH}}}
\def\DR{{\mathrm{DR}}}
\def\Per{{\mathrm{Per}}}

\def\Diffc{{\mathsf{Diffc}}}
\def\Stokes{{\mathsf{St(\scr{A}_\per)}}}
\def\preSt{{\mathsf{St}^{\rm pre}(\scr{A}_\per)}}

\def\aut{{\scr{A}ut}}

\title[Stokes structures of mild difference modules]
{Stokes structure of mild difference modules}
\author{Yota Shamoto }

\begin{document}
\begin{abstract}
    We introduce a category of filtered sheaves 
    on a circle to
    describe the Stokes phenomenon 
    of linear difference equations with mild singularity.
    The main result is a mild difference analog of 
    the Riemann-Hilbert correspondence 
    for germs of meromorphic connections 
    in one complex variable 
    by Deligne-Malgrange. 
\end{abstract}
\maketitle

\section{Introduction}
\subsection{Mild difference modules}
Let $\scr{O}_t\coloneqq\C\{t\}$ be the ring of 
convergent power series in a variable $t$.
Let $\mathscr{O}_t(*0)=\C\conv{t}$ be the 
quotient field. 
Let 
$\phi_t
$
be an automorphism on $\scr{O}_t(*0)$ defined as
$\phi_t(f)(t)\coloneqq f(\tfrac{t}{1+t})$.
If we set $s=t^{-1}$, 
we have $\phi_t(f)(s)=f(s+1)$. 
By \textit{a difference module} 
(over 
the difference field $(\scr{O}_t(*0),\phi_t)$),
we mean a
pair $(\scr{M},\psi)$
of a finite-dimensional 
$\scr{O}_t(*0)$-vector space $\scr{M}$ and 
an automorphism $\psi\colon \scr{M}\to \scr{M}$
of $\C$-vector spaces satisfying the relation
$\psi(fv)=\phi_t(f)\psi(v)$
for any $f\in \scr{O}_t(*0)$ and $v\in \scr{M}$.

There is a class of difference modules 
called \textit{mild} \cite{Galois}*{\S 9}.
A difference module is called \textit{mild} if 
it is isomorphic to a module of the form
$(\mathscr{O}_t(*0)^{\oplus r},
A(t)\phi_t^{\oplus r })$
where $A(t)$ has entries in $\mathscr{O}_t$
and the constant term $A(0)$ is invertible.

The purpose of this paper 
is to establish the Riemann-Hilbert correspondence for 
mild difference modules as an analog of that for 
germs of meromorphic connections 
in one complex variable 
by Deligne-Malgrange
\cites{Deligne,Malgrange}. 
See also \cite{Sabbah}. 

\subsection
{Stokes filtered
locally free sheaves
for mild difference modules}
To formulate the Riemann-Hilbert correspondence for
mild difference modules, 
we introduce the notion of 
Stokes filtered locally free sheaves 
for 
difference modules in \S \ref{SSt}. 
We explain the notion briefly, 
comparing with the 
case of meromorphic connections. 

In the case of germs of meromorphic connections, 
we consider the notion of 
a Stokes filtered local system on 
$S^1=\{z\in\C\mid |z|=1\}$.
It is a pair 
$(\cal{L},\cal{L}_\bullet)$
of a local system $\cal{L}$ of 
finite-dimensional $\C$-vector spaces on $S^1$
and a filtration $\cal{L}_{\bullet}$ 
on $\cal{L}$ indexed by a 
sheaf 
$\cal{I}=\bigcup_{m\geq 0}
z^{-\frac{1}{m}}\C[z^{-\frac{1}{m}}]$ 
of ordered abelian groups. 

The Deligne-Malgrange theorem claims that 
there is an equivalence 
(called the Riemann-Hilbert functor)  between 
the category of germs of meromorphic 
connections 
and the category of
Stokes filtered local systems
on $S^1$. 

When a Stokes filtered local system 
$(\cal{L},\cal{L}_\bullet)$
corresponds to germs of a meromorphic connection 
$\cal{M}$ by the Riemann-Hilbert functor,  
the sheaf $\cal{L}$ is regarded as 
the sheaf of flat sections of $\cal{M}$ on sectors
and 
the filtration $\cal{L}_{\bullet}$ 
describes the growth rate of 
the sections. The filtration $\cal{L}_\bullet$ 
is called 
the Stokes filtration on $\cal{L}$
since the relation between the splittings 
of $\cal{L}_\bullet$ on different domains 
describes the classical Stokes phenomenon
of the solutions of the differential equation 
associated to $\cal{M}$. 

In the case of difference modules, 
we consider a locally free sheaf $\scr{L}$
over a sheaf $\scr{A}_\per$
of rings over $S^1$. 
Here, $\scr{A}_\per$
is a sheaf of rings over $S^1$
defined as follows: 
\begin{align*}
    \scr{A}_\per(U)=
    \begin{cases}
    \C\conv{u^{-1}}&
    (U\subset (0,\pi),U\neq \emptyset)\\
    \C\conv{u}&(U\subset (-\pi,0),U\neq \emptyset)\\
    \C[u^{\pm 1}]
    &(U\cap\{e^{i\pi},e^0\}\neq\emptyset),
    \end{cases}
\end{align*}
where $U$ is assumed to be connected
and we set $(a,b)\coloneqq 
\{e^{\i\theta}\in S^1\mid a<\theta<b\}$ 
for $a,b\in \R$ with $a<b$.
If we put $u=\exp(2\pi\i t^{-1})$, 
we can regard $\scr{A}_\per$ as
a sheaf of rings of a certain 
class of  $\phi_t$-invariant 
(or, periodic with respect to  
$s\mapsto s+1$) functions 
(see \S \ref{Prel} for more details). 

Then, we define
a filtration $\scr{L}_{\bullet}$
on $\scr{L}$
indexed by a sheaf of ordered 
abelian groups.
We call it a Stokes filtration. 
It will be turned out 
that the filtration describes 
the growth rate of the solutions of 
the difference equation associated 
to the difference module. 
A new feature
of the filtration is the compatibility 
of the action of $u$ with the filtration:
\begin{align*}
    u\scr{L}_{\leqslant \frak{a}}
    =\scr{L}_{\leqslant \frak{a}+2\pi\i t^{-1}},
\end{align*}
where $\frak{a}$ is an arbitrary index
(see \S \ref{SSt} for more details).

It is worth mentioning that the Stokes filtered 
$\scr{A}_\per$-module
can be  non-graded
even if it
is rank one as a free 
$\scr{A}_\per$-module. 
This point will be explained in \S \ref{R1}.

\subsection{Main result}
Let $\Diffc^\mild$ be the category of 
mild difference modules.
Let $\Stokes$ be the category of 
the Stokes filtered locally free 
$\scr{A}_\per$-modules. 
Then, we can state the main result of the
present paper:
\begin{theorem}[Theorem \ref{Main}]\label{IMa}
There is a functor 
$\RH\colon \Diffc^\mild \to \Stokes$,
which is an equivalence of categories.
\end{theorem}
This result is analogous to
that of Deligne-Malgrange 
\cites{Deligne,Malgrange}
(See 
\cite{Sabbah}*{Theorem 5.8}). 
The proof of this theorem is 
similar to that can be found 
in \cite{Sabbah}. 
The main difference is 
the definition of the functor 
$\RH$. 
See Remark \ref{Periodic} 
for more precise.

The author hopes that 
this result contributes 
to the intrinsic understanding of 
linear difference modules. 
In particular, 
it would be interesting to
use the result
to
describe the Stokes structure of
the Mellin transformation 
of a holonomic $\cal{D}$-module
concerning
recent progress 
\cites{Bloch, Local, GS,
garcia2018mellin} 
in the study 
of the Mellin transformations
(see Remark \ref{Mellin}).

\subsection{Outline of the paper}
In \S \ref{Prel},
we prepare some notions used throughout the paper.
In \S \ref{SSt},
we introduce the notion of Stokes filtered 
$\scr{A}_\per$-modules. 
In \S \ref{RH&RS}, 
we formulate and prove the main theorem
assuming a theorem
proved in \S \ref{PRH}. 
\subsection*{Acknowledgement}
The author would like to express his deep gratitude to 
Claude Sabbah, who gave the author fruitful comments on 
preliminary versions of this paper. 
The author also would like to thank
Tatsuki Kuwagaki,
Takuro Mochizuki, 
Fumihiko Sanda, and Takahiro Saito 
for discussions and encouragement 
in many occasions.
 The author is supported by JSPS KAKENHI Grant Number JP
 20K14280. 
 \section{Preliminaries}\label{Prel}
In this section,
we prepare some notions used throughout the paper. 
\subsection{An automorphism on a projective line }\label{Notations}
Let $\C$ be the set of complex numbers.
Set $\C^*\coloneqq \C\setminus \{0\}$. 
Natural inclusion is denoted by 
$\jmath\colon \C^*\to \C$. 
When we distinguish a variable such as $t$,
we use the symbols $\C_t$ and $\C^*_t$. 
Let 
    $S^1=\{e^{\i\theta}\in\C\mid \theta\in \R\}$
be the unit circle, 
where we set $\i=\sqrt{-1}$. 
For two real numbers $a,b$ with $a<b$, 
we set $(a,b)\coloneqq \{e^{\i\theta}\in S^1\mid a<\theta<b\}$.
\subsubsection{Real blowing up}\label{RBU}
We set
\begin{align*}
    \widetilde{\C}=\{(t,e^{\i\theta})\in \C\times S^1\mid 
    t=|t|e^{\i\theta}\},
\end{align*}
which is called \textit{the real blowing up of $\C$ at the origin.}
When we distinguish a variable such as $t$, 
we use the notation $\widetilde{\C}_t$. 
There are maps
    $\varpi\colon \widetilde{\C}\longrightarrow \C$,
    $\widetilde{\jmath}\colon \C^*\hookrightarrow 
    \widetilde{\C}$,
    and $\widetilde{\imath}\colon S^1 \hookrightarrow \widetilde{\C}$
defined by $\varpi(t,e^{\i\theta})=t$, $\widetilde{\jmath}(t)=(t,t/|t|)$,
and $\widetilde{\imath}(e^{\i\theta})=(0,e^{\i\theta})$, respectively. 
We sometimes denote the boundary of $\widetilde{\C}_t$
by $S^1_t$ to distinguish a variable 
such as $t$.

\subsubsection{Unit disc}\label{Unit}
Let $\Delta\coloneqq\{t\in\C\mid |t|<1\}$
be a unit open disc.
We set $\Delta^*\coloneqq \Delta\setminus \{0\}$ and 
    $\widetilde{\Delta}\coloneqq \{(t,e^{\i\theta})\in \widetilde{\C}
    \mid t\in \Delta\}.$
Let $\varpi_\Delta\colon \widetilde{\Delta}\to\Delta$
be the projection. 
When we distinguish a variable such as $t$,
we use the notations $\Delta_t$, $\Delta_t^*$ and $\widetilde{\Delta}_t$. 
The natural inclusions 
are denoted by 
$\jmath_\Delta\colon \Delta^*\to \Delta$, 
$\widetilde{\imath}_\Delta\colon S^1\to \widetilde{\Delta}$,
and 
$\widetilde{\jmath}_\Delta\colon 
\Delta^*\to \widetilde{\Delta}$.
%
%
Let $\varphi_t\colon \Delta_t\to \C_t$ be a holomorphic function
defined as
\begin{align*}
    \varphi_t(t)\coloneqq \frac{t}{1+t}.  
\end{align*}
The map uniquely extends
to a continuous map 
$\widetilde{\varphi}_t\colon \widetilde{\Delta}_t
\longrightarrow \widetilde{\C}_t$.

\subsubsection{Another coordinate}\label{Another}
Let $\C_s$ be a complex plane 
with a coordinate $s$. 
When we use 
the two complex variables $s$ and $t$,
we implicitly assume the relation
\begin{align}\label{s=t}
    s=t^{-1}. 
\end{align}
In other words, 
we consider the complex projective line $\P^1$
covered by two open subsets $\C_s$ and $\C_t$
with the relation \eqref{s=t}. 
Let $\varphi_s\colon \C_s\to \C_s$
be the map defined as
\begin{align*}
    \varphi_s(s)=s+1. 
\end{align*}
Under the  relation \eqref{s=t}, 
the map $\varphi_s$ 
coincides with the map $\varphi_t$ 
on the domain $\Delta_t^*=\{s\in \C_s \mid |s|>1\}$. 
Hence the maps $\varphi_s$ and $\varphi_t$ are glued to 
an automorphism on $\P^1$.

\subsection{Sheaves of periodic functions} 
For a point $x$ in a topological space 
$X$ and a sheaf $\scr{F}$
on $X$, 
let $\scr{F}_x$ denote the set of germs of $\scr{F}$ at $x$
(with some structure).
For a continuous map $f\colon X\to Y$ between topological spaces, 
$f_*$ denotes the pushing forward of sheaves 
and $f^{-1}$ denotes the pull back of sheaves. 
If
$X$ is a complex manifold,
let $\scr{O}_X$ denote the sheaf of holomorphic functions on $X$. 
If $X$ is a Riemann surface
and $D$ is a finite set of points,
then let $\scr{O}_X(*D)$ denote the sheaf
of meromorphic functions on $X$
whose poles are contained in $D $.

\subsubsection{Functions with fixed asymptotic behavior}
\label{Sheaves}
Using the notations in \S \ref{RBU}, set
\begin{align*}
    \widetilde{\scr{O}}
    \coloneqq 
    \widetilde{\imath}^{-1}\widetilde{\jmath}_*\scr{O}_{\C^*},
\end{align*}
which is a sheaf on $S^1$. 
There are subsheaves $\scr{A}^{\leqslant 0}$, and $\scr{A}^{<0}$
in $\widetilde{\scr{O}}$ characterized by 
their asymptotic behavior as follows 
(see \cite{Sabbah} for precise definitions):
\begin{itemize}
    \item $\scr{A}^{\leqslant 0}$ is the sheaf of holomorphic functions 
    \textit{which are of moderate growth.} 
    \item $\scr{A}^{<0}$ is the sheaf of holomorphic functions 
    \textit{which are of rapid decay}. 
\end{itemize}
To emphasize the coordinate function such as $t$,
 we use the notation $\widetilde{\scr{O}}_t$, $\scr{A}_t^{\leqslant 0}$, e.t.c.

\subsubsection{Periodic functions}\label{PER}
The map
$\widetilde{\varphi}_t\colon\widetilde{\Delta}_t
\to \widetilde{\C}_t$ 
defined in \S \ref{Unit} naturally
induces a morphism 
   $ \widetilde{\varphi}_t^*\colon \widetilde{\jmath}_*\scr{O}_{\C^*}
    \longrightarrow
    \widetilde{\varphi}_{t*}\widetilde{\jmath}_{\Delta*}\scr{O}_{\Delta^*}$
defined as the composition with the map $\varphi$ restricted to
a suitable open subset. 
We then set 
\begin{align*}
    \widetilde{\phi}_t\coloneqq
    \widetilde{\imath}^{-1}(\widetilde{\varphi}_t^*)
    \colon \widetilde{\scr{O}}_t\to \widetilde{\scr{O}}_t,
\end{align*}
where we used the relation
$\widetilde{\varphi}_t\circ\widetilde{\imath}_\Delta=\widetilde{\imath}$. 
Note that 
the subsheaves
$\scr{A}^{\leqslant0}_t$
and $\scr{A}^{<0}_t$
are invariant under the automorphism $\widetilde{\phi}_t$.
We then set
$\nabla_{\widetilde{\phi}_t}\coloneqq \widetilde{\phi}_t-\id_{\widetilde{\scr{O}}}.$
The restrictions 
of $\nabla_{\widetilde{\phi}_t}$
to  $\scr{A}^{\leqslant 0}_t$ and $\scr{A}^{<0}_t$
are also denoted by the same symbol.
Set 
\begin{align}\label{u}
    u\coloneqq \exp(2\pi\i s)=\exp(2\pi\i t^{-1})
\end{align}
and $v=u^{-1}$. 
Then, we have the equality $\nabla_{\widetilde{\phi}_t}(u)=0$. 

\begin{definition}[Sheaves of periodic functions]
Let $\widetilde{\scr{O}}_{\per}$,
$\scr{A}^{< 0}_\per$,
and ${\scr{A}}_\per^{\leqslant 0}$
denote the kernel of 
the operator $\nabla_{\widetilde{\phi}_t}$ on 
$\widetilde{\scr{O}}$,
$\scr{A}^{< 0}$
and ${\scr{A}}^{\leqslant 0}$, respectively. 
\end{definition}

\begin{lemma}[c.f. {\cite{Galois}*{p.117-118}}]
For a non-empty connected open subset 
$U\subset S^1_t$, 
we have the following descriptions of 
$\widetilde{\scr{O}}_\per(U)$, 
$\scr{A}^{\leqslant 0}_\per(U)$, 
and
$\scr{A}^{<0}_\per(U)$$:$ 
\begin{itemize}
    \item If $U\subset (0,\pi)$, then 
    \begin{align*}
        \widetilde{\scr{O}}_\per(U)=(\jmath_*\scr{O}_{\C_{v}^*})_0,
        \quad
        \scr{A}_\per^{\leqslant 0}(U)=\scr{O}_{v},
        \quad 
        \scr{A}_\per^{<0}(U)=v\scr{O}_{v}.
    \end{align*}
    \item If $U\subset (-\pi,0)$, then
    \begin{align*}
        \widetilde{\scr{O}}_\per(U)=(\jmath_*\scr{O}_{\C_{u}^*})_0,
        \quad
        \scr{A}_\per^{\leqslant 0}(U)=\scr{O}_{u},
        \quad 
        \scr{A}_\per^{<0}(U)=u\scr{O}_{u}.
    \end{align*}
    \item If $e^0\in U$ or $e^{\pi\i}\in U$, then 
    \begin{align*}
        \widetilde{\scr{O}}_\per(U)=\scr{O}_{\C_{u}}(\C_u^*),
        \quad
        \scr{A}_\per^{\leqslant 0}(U)=\C,
        \quad 
        \scr{A}_\per^{<0}(U)=0.
    \end{align*}
\end{itemize}
Here, the rings $(\jmath_*\scr{O}_{\C_u})_0$, 
$(\jmath_*\scr{O}_{\C_v})_0$,
$\scr{O}_{u}\coloneqq\scr{O}_{\C_u,0}$, 
and $\scr{O}_{v}\coloneqq\scr{O}_{\C_v,0}$
are naturally embedded into $\widetilde{\scr{O}}(U)$
in each case
via the relations \eqref{u} and $v=u^{-1}$. 
\end{lemma}
\begin{proof}
For a section 
$[f]\in \widetilde{\scr{O}}_\per(U)$,
there exists a representative
$f\in\widetilde{\scr{O}}(\cal{U})$ 
on an open subset 
$\cal{U}\subset \widetilde{\C}$ 
with $\del\widetilde{\C}\cap \cal{U}=U$. 
We may assume that we have
$f(s)=f(s+1)$ for 
$s\in \cal{V}\coloneqq 
\cal{U}\cap\widetilde{\varphi}^{-1}
(\cal{U})$.
It then follows 
that 
$f$ can be regarded as a 
holomorphic function on
$\exp(2\pi\i(\cal{V}\setminus U))\subset\C^*$. 
Since we have
\begin{align*}
&|u|=\exp(2\pi|s|\Re(\i e^{\i\arg (s)})),
\\
&\arg(u)=|s|\Im (2\pi\i e^{\i\arg (s)}) \mod 2\pi
\end{align*}
for $u=\exp(2\pi\i s)$,
if $U$ is an open interval in $(0,\pi)$, 
then the family 
$\{\infty\}\cup\exp(2\pi\i(\cal{U}\setminus U))$
of open subsets for 
$\cal{U}\subset \widetilde{\C}$ 
with $\del\widetilde{\C}\cap \cal{U}=U$
defines a fundamental system of neighborhoods around the
infinity of the $u$-plane,
which implies $\widetilde{\scr{O}}_\per(U)
=(\jmath_*\scr{O}_{\C_{v}^*})_0$. 
Since $v=u^{-1}$ rapidly decays on $\cal{U}$,
we also have 
$\scr{A}_\per^{\leqslant 0}(U)=\scr{O}_{v}$, 
and $\scr{A}_\per^{<0}(U)=v\scr{O}_{v}$. 
The other part of the lemma can be proved in 
a similar way. We left them to the reader.   
\end{proof}

\begin{definition}
We set 
   $ \scr{A}_\per\coloneqq \sum_{n\in\Z}u^n\scr{A}^{\leqslant 0}_\per 
   \subset\widetilde{\scr{O}}$.
\end{definition}
Let $\C[u^{\pm 1}]$ denote the ring of Laurent polynomials in $u$.
\begin{corollary}
For a connected open subset $\emptyset\neq U\subset S^1_t$, 
we have the following:
\begin{align*}
    \scr{A}_\per(U)=\begin{cases}
    \scr{O}_v(*0) & (U\subset (0,\pi))\\
    \scr{O}_u(*0) & (U\subset (-\pi,0))\\
    \C[u^{\pm 1}] & (U\cap\{e^0,e^{\pi\i}\}\neq \emptyset)
    \end{cases}
\end{align*}
where we set $\scr{O}_{w}(*0)\coloneqq \scr{O}_{\C_w}(*\{0\})_0$
for $w=u,v$.  \qed
\end{corollary}

\subsection{Finite maps}\label{Ram}
Fix a positive integer $m$. 
Let $\rho\colon \C_\tau\to \C_t$ be the $m$-th power map. 
In other words, we consider another parameter $\tau$ with the relation
\begin{align}
    \tau^m=t. 
\end{align}
There exists a real blow-up 
$\widetilde{\rho}\colon \widetilde{\C}_\tau
    \longrightarrow \widetilde{\C}_t$
of $\rho$ defined as $\widetilde{\rho}(\tau,e^{\i \vartheta})\coloneqq 
(\rho(\tau),e^{\i m\vartheta})$. 
The restriction of $\widetilde{\rho}$
to the boundary is denoted by ${\rho}\colon S^1_\tau \to S^1_t$
or $\rho_m$ if we emphasize the dependence on $m$. 
The multiplicative group of the $m$-th roots of unity
\begin{align*}
    \mu_m\coloneqq \{\zeta\in\C\mid \zeta^m=1\}. 
\end{align*}
is regarded as the group of automorphisms on $S^1_\tau$ over $S^1_t$
in a natural way. 
Let $\sigma_\zeta\colon S^1_\tau\to S^1_\tau$
be the automorphisms corresponding to $\zeta\in\mu_m$.

\section{Stokes filtered locally free sheaves}\label{SSt}
In this section, 
we introduce the notion of a
Stokes filtered locally free sheaves over 
the sheaf $\scr{A}_\per$ of rings, 
which will be called a 
 Stokes filtered 
$\scr{A}_\per$-modules.
Concrete examples of Stokes filtered 
$\scr{A}_\per$-modules will be given 
in the next section.

\subsection{Sheaves of indexes}
In this subsection,
we prepare some sheaves of ordered abelian groups. 
They will be used to 
define filtrations on $\scr{A}_\per$-modules. 

\subsubsection{A sheaf on $S^1_\tau$}
Fix $m\in \mathbb{Z}_{>0}$. 
Recall the notations in \S \ref{Ram}
such as $\tau$ 
 and $S^1_\tau$
corresponding to $m$.
We recall that there is a sheaf 
$\widetilde{\scr{O}}_{\tau}$ on ${S}^1_{\tau}$ defined in \S \ref{Sheaves}.

\begin{definition}\label{qtau}
Let $\scr{I}_{m,\tau}\subset \widetilde{\scr{O}}_\tau$
be a subsheaf of $\C$-vector spaces
locally generated by the sections 
represented by
the functions of the form
\begin{align}\label{q(tau)}
    \frak{a}(\tau)=\sum_{\ell=1}^mc_\ell \tau^{-\ell}
\end{align}
where  
$c_1,\dots,c_m\in \C$. 
\end{definition}
\begin{remark}
    The motivation for introducing the sheaf
    $\scr{I}_{m,\tau}$ 
    comes from 
    the definition
    of mild difference modules (Definition \ref{E&U}). 
\end{remark}

\subsubsection{Sheaves on $S^1_t$}
Let $\rho\colon S^1_\tau\to S^1_t$ be the map defined in 
\S \ref{Ram}. 
Let 
$\rho^*\colon \widetilde{\scr{O}}_t\to \rho_*\widetilde{\scr{O}}_\tau$
be the adjoint morphism. 
We regard $\widetilde{\scr{O}}_t$ as a 
subsheaf of $\rho_*\widetilde{\scr{O}}_\tau$ 
by this $\rho^*$. 
\begin{definition}
For a positive integer $m$, we set
\begin{align*}
    \scr{I}_{m}
    \coloneqq \rho_{*}\scr{I}_{m,\tau}\cap \widetilde{\scr{O}}_t
    \subset\rho_*\widetilde{\scr{O}}_\tau.
\end{align*}
We then set 
$\scr{I}\coloneqq \bigcup_{m=1}^\infty \scr{I}_m. $
\end{definition}

\subsubsection{Orders on local sections}
We shall define partial orders 
on the space of sections of sheaves introduced above. 
\begin{definition}
For an open subset
$U\subset S^1_t$ 
and sections $\frak{a},\frak{b}\in \scr{I}(U)$, 
we define 
\begin{align*}
    \frak{a}\leqslant_U\frak{b}&\Longleftrightarrow 
    \exp(\frak{a}-\frak{b})\in \scr{A}^{\leqslant 0}(U), \text{ and }\\
    \frak{a}<_U\frak{b}&\Longleftrightarrow 
    \exp(\frak{a}-\frak{b})\in \scr{A}^{< 0}(U).
\end{align*} 
\end{definition}

\subsection{Stokes filtered $\scr{A}_\per$-modules}
\subsubsection{Pre-Stokes filtrations}\label{DefSt}
We shall define the notion of a pre-Stokes filtration on
a $\scr{A}_{\per}$-module as follows. 
\begin{definition}\label{Pre-St}
Let $\scr{L}$ be a $\scr{A}_\per$-module.  
\textit{A pre-Stokes filtration} on 
$\scr{L}$ 
is a family 
\begin{align*}
    \scr{L}_\bullet\coloneqq 
    \{\scr{L}_{\leqslant \frak{a}}\subset \scr{L}_{|U}\mid
    U\subset S^1\colon\text{open subset},\ \frak{a}\in \scr{I}(U)\}
\end{align*}
of $\scr{A}_{\per|U}^{\leqslant 0}$-submodules in  
$\scr{L}_{|U}$ for all open subsets $U$ with the following properties:
\begin{enumerate}
    \item If $\frak{a}_{|V}=\frak{b}$ for $V\subset U$, $\frak{a}\in \scr{I}(U)$, and 
    $\frak{b}\in \scr{I}(V)$,
    then 
    $\scr{L}_{\leqslant \frak{a}|V}=\scr{L}_{\leqslant \frak{b}}.$
    \item If $\frak{a}\leqslant_U\frak{b}$ for 
          $\frak{a},\frak{b}\in \scr{I}(U)$,
          then 
          $\scr{L}_{\leqslant \frak{a}}\subset 
          \scr{L}_{\leqslant \frak{b}}.$
    \item For any $n\in \Z$ and $\frak{a}\in \scr{I}(U)$,
          we have
              $u^n\scr{L}_{\leqslant \frak{a}}
              =\scr{L}_{\leqslant \frak{a}+2n\pi\i s}.$ 
\end{enumerate}
\end{definition}
A pair $(\scr{L},\scr{L}_\bullet)$ of a $\scr{A}_{\per}$-module
and a pre-Stokes filtration on it is called
\textit{a pre-Stokes filtered $\scr{A}_{\per}$-module}. 
A \textit{morphism} between two pre-Stokes filtered 
$\scr{A}_{\per}$-modules 
$(\scr{L},\scr{L}_\bullet)$, 
and $(\scr{L}',\scr{L}_\bullet')$
is a morphism $\lambda\colon \scr{L}\to\scr{L}'$
of $\scr{A}_{\per}$-modules
such that 
$\lambda_{|U}(\scr{L}_{\leqslant\frak{a}})\subset 
\scr{L}'_{\leqslant \frak{a}}$
for any $\frak{a}\in \scr{I}{(U)}$. 

\subsubsection{Grading}
For each $\frak{a}\in\scr{I}(U)$
and a 
pre-Stokes filtered $\scr{A}_{\per}$-module
$\scr{L}=(\scr{L},\scr{L}_\bullet)$,
we set
\begin{align*}
    \scr{L}_{<\frak{a}}
    \coloneqq 
    \sum_{\frak{b}<_U\frak{a}}
    \scr{L}_{\leqslant \frak{b}}
\end{align*}
and 
$\gr_\frak{a}\scr{L}\coloneqq 
\scr{L}_{\leqslant \frak{a}}/\scr{L}_{<\frak{a}}$.
By condition (3) in the Definition,  
we have the inclusion 
$\scr{A}^{<0}_{|U}\cdot 
\scr{L}_{\leqslant \frak{a}}\subset 
\scr{L}_{<\frak{a}}.$ 
Hence it is natural to regard 
$\gr_\frak{a}\scr{L}$ as 
a module over   
$\scr{A}^{\leqslant 0}_{\per|U}/\scr{A}_{\per|U}^{<0} 
=\C_U$.
Again by condition (3),
we have an isomorphism 
\begin{align}\label{UN}
    u^n\colon \gr_{\frak{a}}\scr{L}\simeqto \gr_{\frak{a}+2\pi\i s n}\scr{L}
\end{align}
for any $n\in \Z$. 
Hence $\bigoplus_{\mathfrak{a}\in \scr{I}(U)}
    \gr_\frak{a}\scr{L}$
is naturally equipped with the structure of a sheaf
of $\C[u^{\pm 1}]$-modules over $U$.  
Then the sheaf $\Sigma(\scr{L})$ of sets
defined as 
\begin{align*}
    \Sigma(\scr{L})(U)\coloneqq 
    \{\frak{a}\in \scr{I}(U)\mid \gr_{\frak{a}}\scr{L}\neq 0\},
\end{align*}
is equipped with the 
the action of the additive group 
$2\pi\i s\Z_{S^1}$ 
in a natural way. 

\begin{lemma}\label{grLemma}
Let $(\scr{L},\scr{L}_\bullet)$ be 
a pre-Stokes filtered 
$\scr{A}_{\per}$-module.
There exists a
unique sheaf $\gr\scr{L}$ of
$\C[u^{\pm 1}]$-modules 
such that for any open subset 
$U\subset S^1$,
 we have
\begin{align*}
    \gr\scr{L}_{|U}=
    \bigoplus_{\mathfrak{a}\in \scr{I}(U)}
    \gr_\frak{a}\scr{L}.
\end{align*}
Moreover, 
$\gr\scr{L}$ is a local system of finitely generated 
free $\C[u^{\pm 1}]$-modules 
if and only if the following conditions 
are satisfied:
\begin{itemize}
    \item For each $\frak{a}\in \scr{I}(U)$, 
          the sheaf $\gr_\frak{a}\scr{L}$
          is a local system of  
          $\C$-vector spaces. 
    \item The sheaf $\Sigma(\scr{L})$
          has locally finitely many 
          $2\pi\i s \Z_{S^1}$-orbits. 
\end{itemize}
\end{lemma}
\begin{proof}
For two connected 
open subsets $V\subset U \subsetneq S^1_t$,  
the restriction map 
$\scr{I}(U)\to \scr{I}(V)$ is an isomorphism.
Then, by condition (1) in Definition \ref{Pre-St}, 
we have
\begin{align*}
    \bigoplus_{\frak{a}\in \scr{I}(U)}\gr_\frak{a}\scr{L}_{|V}
    =\bigoplus_{\frak{b}\in \scr{I}(V)}\gr_{\frak{b}}\scr{L}. 
\end{align*}
The first assertion follows from this equality. 
By \eqref{UN}, 
we have a natural isomorphism 
\begin{align*}
    \gr_\frak{a}\scr{L}\otimes\C[u^{\pm 1}]
    \simeq \bigoplus_{n\in \Z} \gr_{\frak{a}+2\pi\i n }\scr{L}
\end{align*}
for each $\frak{a}\in \scr{I}{(U)}$.
The second assertion follows from this isomorphism. 
\end{proof}
\subsubsection{Stokes filtration}
\begin{definition}
Let $\scr{L}$ be a $\scr{A}_\per$-module. 
A pre-Stokes filtration $\scr{L}_\bullet$
on $\scr{L}$ is called 
\textit{a Stokes filtration}
if the following conditions are satisfied: 
\begin{enumerate}
    \item The graded sheaf $\gr\scr{L}$ is a local system of 
    finitely generated
    $\C[u^{\pm 1}]$-modules.
    \item 
    For each point $x\in S^1$, there exist a
    neighborhood $U$ of $x$ and 
    an isomorphism
    \begin{align*}
        \eta\colon \scr{A}_{\per|U}\otimes_{\C[u^{\pm 1}]}
        \gr\scr{L}_{|U}\overset{\sim}{\longrightarrow}
        \scr{L}_{|U}
    \end{align*}
    of filtered $\scr{A}_{\per|U}$-modules
    such that $\gr(\eta)=\id_{\gr\scr{L}_{|U}}$. 
\end{enumerate}
A pair $(\scr{L},\scr{L}_\bullet)$ of a $\scr{A}_{\per}$-module
and a Stokes filtration on it is called
\textit{a Stokes filtered $\scr{A}_{\per}$-module}. 
The category of Stokes filtered $\scr{A}_{\per}$-module
is denoted by $\Stokes$,
which is defined to be a full subcategory of
the category of pre-Stokes filtered $\scr{A}_{\per}$-modules. 
\end{definition}
\begin{remark}
    By Lemma \ref{grLemma}, 
    the sheaf $\scr{L}$ should be finitely generated and
    locally free over $\scr{A}_\per$ if it admits a Stokes filtration. 
\end{remark}
\subsubsection{Graded Stokes filtered 
$\scr{A}_{\per}$-modules 
and some operations}
\begin{definition}
A Stokes filtered $\scr{A}_\per$-module $(\scr{G},\scr{G}_\bullet)$ 
is called \textit{graded}
if we have
an isomorphism 
\begin{align*}
    \eta\colon \scr{G}\overset{\sim}{\longrightarrow}
    \scr{A}_\per\otimes\gr(\scr{G})
\end{align*}
of Stokes filtered $\scr{A}_\per$-modules such that 
$\gr(\eta)=\id$. 
\end{definition}

For two (pre-)Stokes filtered $\scr{A}_\per$-modules
$(\scr{L}_1,\scr{L}_{1,\bullet})$
and 
$(\scr{L}_2, \scr{L}_{2,\bullet})$
we can define the (pre-)Stokes filtrations
on the tensor product
$\scr{L}_1\otimes \scr{L}_2$
and the sheaf of internal-homs
$\homscr(\scr{L}_1,\scr{L}_2)$
as follows:
\begin{align*}
    (\scr{L}_1\otimes \scr{L}_2)_{\leqslant \frak{a}}
    &\coloneqq\sum_{\frak{b}\in\scr{I}(U)} 
    \scr{L}_{1,\leqslant\frak{b}}
    \otimes_{\scr{A}_\per^{\leqslant 0}}
    \scr{L}_{2,\leqslant\frak{a}-\frak{b}},\\
    \homscr(\scr{L}_1,\scr{L}_2)_{\leqslant\frak{a}}
    &\coloneqq 
    \sum_{\frak{b}\in\scr{I}(U)}
    \homscr_{\scr{A}_\per^{\leqslant 0}}(\scr{L}_{1,\leqslant \frak{b}},
    \scr{L}_{2,\leqslant \frak{a}+\frak{b}}).
\end{align*}
We also use the notation
$\scr{E}nd(\scr{L})\coloneqq \homscr(\scr{L},\scr{L})$.

\subsection{Classification}
 For a Stokes $\scr{A}_{\per}$-module
 $\scr{L}=(\scr{L},\scr{L}_\bullet)$,
we set  
\[\aut^{<0}(\scr{L})\coloneqq 
\id_{\scr{L}}+\scr{E}nd({\scr{L}})
_{<0}.\]
\begin{theorem}\label{ClSt}
Let $\scr{G}$ be a graded Stokes filtered 
$\scr{A}_\per$-module.
Then, there is a natural 
one-to-one correspondence 
between the cohomology set 
\begin{align*}
    H^1(S^1, 
    \aut^{<0}
    (\scr{G}))
\end{align*}
and the set of isomorphism classes of 
pairs 
$((\scr{L},\scr{L}_{\bullet}), \Xi_{\gr} )$ of
\begin{itemize}
\item
 a Stokes filtered 
    $\scr{A}_{\per}$-module
    $\scr{L}=(\scr{L},\scr{L}_{\bullet})$,
    and 
\item an isomorphism 
    $\Xi_\gr\colon \scr{A}_\per\otimes_{\C[u^{\pm 1}]} \gr\scr{L}
    \simeqto\scr{G}$ of Stokes filtered $\scr{A}_\per$-modules. 
\end{itemize}
\end{theorem}
\begin{proof}
Since this is standard,
we only give the construction of the cohomology class. 
Let $((\scr{L},\scr{L}_{\bullet}), \Xi_{\gr} )$
be the pair in the statement. 
We can take a finite open covering 
$S^1=\bigcup_{\alpha\in \Lambda} U_\alpha$ 
such that 
\begin{itemize}
\item 
$\Lambda=\Z/N\Z$ for a positive integer $N$, 
\item
$U_\alpha$ is an open interval for each $\alpha\in \Lambda$, and
\item
$U_\alpha\cap U_\beta=\emptyset$ for 
$\beta\notin \{\alpha-1,\alpha,\alpha+1\}$
\end{itemize}
and isomorphisms
\begin{align*}
    \eta_\alpha \colon \rho^{-1}
        \scr{A}_{\per|U_\alpha}
        \otimes_{\C[u^{\pm 1}]}
        \gr\scr{L}_{|U_\alpha}
        \overset{\sim}{\longrightarrow}
        \scr{L}_{|U_\alpha}
\end{align*}
of filtered $\rho^{-1}\scr{A}_{\per|U_\alpha}$-modules
    such that $\gr(\eta_\alpha)=
    \id_{\gr^{}\scr{L}_{|U_\alpha}}$. 
Then, the tuple 
\[\left(\Xi_{\gr|U_{\alpha,\alpha+1}}
\circ 
\eta_{\alpha+1|U_{\alpha,\alpha+1}}^{-1}
\circ
\eta_{\alpha|U_{\alpha,\alpha+1}}
\circ 
\Xi^{-1}_{\gr|U_{\alpha,\alpha+1}}
\right)_{\alpha\in \Lambda}\]
defines a class in 
$H^1(S^1, 
    \aut^{<0}
    (\scr{G}))$,
    where we put $U_{\alpha,\alpha+1}=U_\alpha\cap U_{\alpha+1}$.
\end{proof}

\section{Riemann-Hilbert correspondence for 
mild
difference modules}
\label{RH&RS}
 In this section,
 we formulate and prove a part of 
Riemann-Hilbert correspondence for
mild difference modules
assuming a vanishing theorem (Theorem \ref{Vanishing}). 

\subsection
{Definition of mild difference modules}
\label{SReg}
We fix some notations and terminologies of 
difference modules and 
recall the definition of 
the mild difference
modules in the sense of 
van der Put and Singer \cite{Galois}.
\subsubsection{Difference modules}
Recall that \textit{a difference ring}
is a pair $(R,\Phi)$
of a commutative ring $R$ with a unit $1$ 
and 
a ring automorphism $\Phi$ on $R$. 
If $R$ is a field, then $(R,\Phi)$
is called a difference field.  
Let $\scr{O}_{t}(*0)
=\scr{O}_{\C_t}(*0)_0=\C\conv{t}$
be the field of convergent Laurent series in $t$. 
Let
$\phi_t\colon \scr{O}_t(*0)\to \scr{O}_t(*0)$
be the automorphism defined by 
$\phi_t(f)(t)\coloneqq f(t(1+t)^{-1})$. 
Then the pair $(\scr{O}_t(*0),\phi_{t})$ is an
example  of a difference field. 
The pair $(\widehat{\scr{O}}_t(*0),\widehat{\phi}_t)$
of the formal completion
$\widehat{\scr{O}}_t(*0)\coloneqq \C\pole{t}$ of $\scr{O}_{t}(*0)$
and the automorphism 
$\widehat{\phi}_t$ defined in a similar way as $\phi_t$
is also an instance of a difference field.

A(n invertible) \textit{difference module}
over a difference ring $(R,\Phi)$
is a pair $(M,\Psi)$ of a 
finitely generated $R$-module $M$
and an automorphism 
$\Psi\colon M\to M$ of 
abelian groups such that 
$\Psi(rm)=\Phi(r)\Psi(m)$
for any $r\in R$ and $m\in M$.
We abbreviate $(M,\Psi)$
by omitting $\Psi$ if there is no 
fear of confusion. 
The category of difference modules
over $(\scr{O}_t(*0),\phi_t)$
is denoted by $\Diffc$, which is an abelian category.

For a difference module
$(\scr{M},\psi)$
over $(\scr{O}_t(*0),\phi_{t})$, 
the automorphism $\psi$
naturally extends to an
automorphism $\widehat{\psi}$ on
$\widehat{\scr{M}}\coloneqq
\widehat{\scr{O}}_t(*0)\otimes_{\scr{O}_t{(*0)}}\scr{M}$. 
The pair $(\widehat{\scr{M}},\widehat{\psi}{})$ 
is called the formal completion of 
$(\scr{M},\psi{})$, which is a difference module over
$\widehat{\scr{O}}_t(*0)\coloneqq \C\pole{t}$. 

\subsubsection{Regular singular 
difference modules}
For a constant matrix
$G\in\End(\C^r)$, we let 
$\scr{R}_{G}=(\scr{O}_{t}(*0)^{\oplus r},\psi_G)$
be a difference module defined by
$\psi_G=(1+t)^{-G}\phi_t^{\oplus r}$, where 
we set 
$(1+t)^{-G}=\exp(-G\log (1+t))$ 
with $\log(1+ t)=-\sum_{n=1}^\infty n^{-1}(-t)^n$. 
The formal completion of $\scr{R}_{G}$
is denoted by $\widehat{\scr{R}}_{G}$. 

\begin{definition}\label{DReg}
A difference module $(\scr{M},\psi)$ over
$(\scr{O}_t(*0),\phi_{t})$
is called \textit{regular singular}
if its formal completion
$(\widehat{\scr{M}},\widehat{\psi})$  
is isomorphic to $\widehat{\scr{R}}_G$
for a matrix $G\in\End(\C^r)$. 
The trivial difference module $0$ 
is regarded as 
a regular singular difference module. 
\end{definition}
\subsubsection{Difference modules 
over an extended field}
\label{Extended}
Fix a positive integer $m$. 
We use the notations in \S \ref{Ram}. 
There is a natural inclusion
$\rho^*\colon \scr{O}_t(*0)\to \scr{O}_\tau(*0)$. 
Regard $t=\tau^m$ as an element of $\scr{O}_\tau(*0)$. 
The pair of the field $\scr{O}_\tau(*0)$
and an automorphism 
$\phi_\tau\colon \scr{O}_\tau(*0)\to \scr{O}_\tau(*0)$
defined as 
$\phi_\tau(\tau)\coloneqq\tau(1+t)^{-1/m}$
forms a difference field. 
The notion of formal completion is 
defined analogously as in the case of 
$(\scr{O}_t(*0),\phi_t)$. 
For a difference module $\scr{M}$ 
over $(\scr{O}_t{(*0)},\phi_t)$,
the pull back 
$\rho^*\scr{M}\coloneqq \scr{O}_{\tau}(*0)\otimes \scr{M}$
is a difference module over $(\scr{O}_{\tau}(*0),\phi_\tau)$
in a natural way. 

\subsubsection{Mild 
exponential factors}\label{ExF}
Fix a positive integer $m$ as in \S \ref{Extended}. 
Let $\frak{a}=\frak{a}(\tau)$ be as \eqref{q(tau)} 
in Definition \ref{qtau}.
We note that 
$\exp(\phi_\tau(\frak{a})-\frak{a})$
is a well defined element in $\scr{O}_{\tau}$.
We define 
a difference module 
$\scr{E}^{\frak{a}}\coloneqq
(\scr{O}_\tau{(*0)},\psi_\frak{a})$
by 
$\psi_\frak{a}(1)\coloneqq
\exp(\phi_\tau(\frak{a})-\frak{a})$.
The formal completion of 
$\scr{E}^{\frak{a}}$ will be denoted by
$\widehat{\scr{E}}^{\frak{a}}{}$. 
For two such difference modules 
$\scr{E}^{\frak{a}}$ and $\scr{E}^{\frak{b}}$,
we have $\scr{E}^{\frak{a}}\simeq \scr{E}^{\frak{b}}$
if and only if $\frak{a}-\frak{b}\in 2\pi\i s\Z$.

\subsubsection{Formal decomposition and 
mild difference modules} 
We recall the definition 
of mild difference modules:

\begin{definition}[{\cite[\S 7.1, p.71]{Galois}}]\label{E&U}
A difference module over
$(\scr{O}_\tau(*0),\phi_\tau)$
 is called
\begin{enumerate}
    \item \textit{mild elementary} if it is a direct sum 
    of the modules of the form 
    $\scr{E}^{\frak{a}}\otimes \rho^*\scr{R}_{G}$, and 
    \item \textit{mild unramified} 
    if it is formally isomorphic to 
    a mild elementary module. 
\end{enumerate}

A difference module $\scr{M}$
over $(\scr{O}_t(*0), \phi_t)$
is called \textit{mild} (resp. \textit{mild graded})
if there exists a positive integer
$m$ such that $\rho^*\scr{M}$
is mild unramified (resp. mild elementary). 
The category of 
mild difference modules 
is denoted by $\Diffc^\mild$. 
\end{definition}
\begin{remark}
The terms ``mild elementary", ``mild unramified", 
and ``mild graded"
cannot be found in \cite{Galois}.
We introduced them to clarify the analogy to 
the theory of differential modules.
 Mild difference modules form
    a special class of difference modules.
By \cite[Lemma 7.4]{Galois}, 
a difference module is mild if and only if 
it is isomorphic to a module 
$(\mathscr{O}_t(*0)^{\oplus r},
A(t)\phi^{\oplus r}_t)$ where 
$A(t)$ has no pole and the constant term 
$A(0)$ is invertible. 
    The most general case 
    is called \textit{wild} in \cite{Galois}. 
\end{remark}
\begin{remark}\label{Mellin}
    Let 
    $\mathbb{M}$ be a holonomic $\cal{D}$-module 
    over $\mathbb{G}_m
    =\mathrm{Spec}(\C[x,x^{-1}])$. 
    The algebraic Mellin transformation 
    $\mathfrak{M}(\mathbb{M})$ of 
    $\mathbb{M}$ is a difference module 
    over the difference ring $(\C[s],\phi_s)$
    where $\phi_s(P(s))=P(s+1)$ for 
    $P(s)\in \C[s]$.
    Then  
    $\mathfrak{M}
    (\mathbb{M})
    \otimes_{\C[s]}\mathscr{O}_t(*0)$
    is mild if 
    $\mathbb{M}$ is regular singular
    at $0$ and $\infty$
    by \cite[Theorem 1, Lemma 3]{garcia2018mellin}.
\end{remark}

\begin{lemma}\label{Graded}
Let $\scr{M}$ be a mild difference module over 
    $(\scr{O}_t(*0),\phi_t)$. 
Let $\scr{N}'$ be a mild elementary module over 
$(\scr{O}_\tau(*0),\phi_\tau)$
such that 
the formal completion $\widehat{\scr{N}}'$ is
isomorphic to 
$\rho^*\widehat{\scr{M}}$. 
Then, there exists a 
mild graded difference module 
$\scr{N}$ over $(\scr{O}_t(*0),\phi_t)$
such that 
$\rho^*\scr{N}$ is isomorphic to $\scr{N}'$. 
\end{lemma}
\begin{proof}
Since the formal completion of   
$\scr{N}'$ is isomorphic to $\rho^*\scr{M}$,
it is naturally equipped with the action of $\mu_m$.
Since $\scr{N}'$ is elementary,
the action naturally lifts to the action on $\scr{N}'$. 
The desired module $\scr{N}$
is the decent of
$\scr{N}'$ by this action. 
\end{proof}
The module $\scr{N}$
in this lemma 
will be called \textit{the graded module of $\scr{M}$}.

\subsection{Riemann-Hilbert functor}\label{RH-fun}
We shall define a functor from 
the category 
$\Diffc^\mild$ of mild difference modules 
over $(\scr{O}_{t}(*0),\phi_t)$
to the category $\Stokes$ of 
Stokes filtered $\scr{A}_{\per}$-modules. 

\subsubsection{Sheaves of difference rings}
The pair $(\widetilde{\scr{O}}_t,\widetilde{\phi}_{t})$
of the sheaf $\widetilde{\scr{O}}_t$ of rings on 
$S^1_t$ and the automorphism 
$\widetilde{\phi}_{t}$ on $\widetilde{\scr{O}}_t$
defined in \S \ref{PER}
is a sheaf of difference rings on $S^1_t$. 
The pairs $(\scr{A}^{\leqslant 0},\widetilde{\phi}_t{})$
and $(\scr{A}^{< 0},\widetilde{\phi}_t{})$ 
are difference subrings in $(\widetilde{\scr{O}}_t,\widetilde{\phi}_{t})$. 

Let  
$\scr{O}_{t}(*0)_{S^1}$
be the constant sheaf on $S^1_t$ with fiber $\scr{O}_{t}(*0)$. 
The automorphism induced from $\phi_t$ on 
$\scr{O}_{t}(*0)_{S^1}$
is denoted by the same letter $\phi_t$. 
We note that the sheaves $\widetilde{\scr{O}}_t$,
$\scr{A}^{\leqslant 0}_t$, and $\scr{A}^{<0}_t$ are
sheaves of difference algebras over 
$(\scr{O}_{t}(*0)_{S^1},\phi_t)$ 
in  a natural way. 
For a difference module $(\scr{M},\psi)$ over 
$(\scr{O}_t(*0),\phi_t)$, 
let $(\scr{M}_{S^1},\psi)$
denote the associated constant sheaf of 
difference modules over 
$(\scr{O}_{t}(*0)_{S^1},\phi_t)$.
For an open subset 
$U\subset S^1$, 
we also use the notation
\[\scr{M}_{|U}\coloneqq 
(\scr{A}^{\leqslant 0}\otimes_{\scr{O}_t(*0)_{S^1}}
\scr{M}_{S^1})_{|U}
\]
for simplicity. 

\subsubsection{Sheaves with 
asymptotic behavior}
For 
$\frak{a}\in \scr{I}(U)$ on an open subset $U\subset S^1_t$, 
let $\scr{A}^{\leqslant \frak{a}}$ (resp. 
$\scr{A}^{< \frak{a}}$) denote the subsheaf
$\exp(\frak{a})\scr{A}^{\leqslant 0}_{|U}$ 
(resp. $\exp(\frak{a})\scr{A}^{< 0}_{|U}$) in 
$\widetilde{\scr{O}}_{|U}$. 
For a local section
$\frak{a}\in \scr{I}(U)$, 
the sheaves $\scr{A}^{\leqslant \frak{a}}$ 
and $\scr{A}^{<\frak{a}}$
are modules over $\scr{A}_{|U}^{\leqslant 0}$. 

\subsubsection{The de Rham complexes 
for difference modules} 

For a sheaf of difference module $(\cal{M},\psi)$
over a sheaf of difference rings, 
we set 
\[\nabla_\psi\coloneqq \psi-\id_\cal{M}.\]

\begin{definition} 
    For a difference module $\scr{M}$
    over $(\scr{O}_t,\phi_t)$, we set
\begin{align*}
    \widetilde{\DR}(\scr{M})
    &\coloneqq 
    [\widetilde{\scr{O}}\otimes\scr{M}_{S^1}
    \xrightarrow{{\nabla}_{\widetilde{\psi}}}
    \widetilde{\scr{O}}\otimes \scr{M}_{S^1}],\\
    \DR_{\leqslant 0}(\scr{M})
    &\coloneqq
    [{\scr{A}}^{\leqslant 0}\otimes \scr{M}_{S^1}
    \xrightarrow{{\nabla}_{\widetilde{\psi}}}
    {\scr{A}}^{\leqslant 0}\otimes \scr{M}_{S^1}],
    \text{ and }\\
    \DR_{< 0}(\scr{M})
    &\coloneqq
    [{\scr{A}}^{< 0}\otimes \scr{M}_{S^1}
    \xrightarrow{{\nabla}_{\widetilde{\psi}}}
    {\scr{A}}^{< 0}\otimes \scr{M}_{S^1}],
\end{align*}
where the automorphisms on 
$\widetilde{\scr{O}}\otimes
\scr{M}_{S^1}$, 
${\scr{A}}^{\leqslant 0}\otimes \scr{M}_{S^1}$,
and ${\scr{A}}^{< 0}\otimes \scr{M}_{S^1}$
are denoted by $\widetilde{\psi}$,
and the  complexes are concentrated on degrees zero and one.
\end{definition}

The proof of the following vanishing theorem 
will be given in the next section:
\begin{theorem}\label{Vanishing}
    Assume that $\scr{M}$ is a mild difference module 
    over $(\scr{O}_t,\phi_t)$, then 
    the complexes $\DR_{\leqslant 0}(\scr{M})$ 
    and $\DR_{< 0}(\scr{M})$
    have non-zero cohomology in degree zero
    at most.    
    Moreover, 
    the natural morphisms
    $\DR_{< 0}(\scr{M})\to 
    \DR_{\leqslant 0}(\scr{M})\to
    \widetilde{\DR}(\scr{M})$
    induce injections
    $\scr{H}^0\DR_{< 0}(\scr{M})\to 
    \scr{H}^0\DR_{\leqslant 0}(\scr{M})\to
    \scr{H}^0\widetilde{\DR}(\scr{M})$. 
\end{theorem}
The following 
corollary will also be proved 
in the next section:
\begin{corollary}\label{Cor}
     Let $\scr{M}$ be a 
        mild difference module 
        and $\frak{a}\in \scr{I}(U)$. 
        If we set
        \begin{align*}
            \DR_{\leqslant \frak{a}}(\scr{M})
            &=
            [\scr{A}^{\leqslant \frak{a}}\otimes
            \scr{M}_{|U}
            \xrightarrow{\nabla_{\widetilde{\psi}}}
            \scr{A}^{\leqslant \frak{a}}\otimes
            \scr{M}_{|U}
            ], \text{ and} \\
            \DR_{< \frak{a}}(\scr{M})
            &=
            [\scr{A}^{< \frak{a}}\otimes
            \scr{M}_{|U}
            \xrightarrow{\nabla_{\widetilde{\psi}}}
            \scr{A}^{< \frak{a}}\otimes
            \scr{M}_{|U}
            ],
            \end{align*}
            then these complexes have non-zero cohomology in degree zero at most. 
       Moreover, 
    the natural morphisms
    $\DR_{< \frak{a}}(\scr{M})\to 
    \DR_{\leqslant \frak{a}}(\scr{M})\to
    \widetilde{\DR}(\scr{M})$
    induce injections
    $\scr{H}^0\DR_{< \frak{a}}(\scr{M})\to 
    \scr{H}^0\DR_{\leqslant \frak{a}}(\scr{M})\to
    \scr{H}^0\widetilde{\DR}(\scr{M})$. 
    The quotient
    \begin{align*}
    \gr_{\frak{a}}\scr{H}^0\DR(\scr{M})\coloneqq
     \scr{H}^0\DR_{\leqslant \frak{a}}(\scr{M})
     /\scr{H}^0\DR_{< \frak{a}}(\scr{M})
    \end{align*}
    is a local system of $\C$-vector space
    over $U$. 
    \end{corollary}
\begin{remark}
    Let $U\subset S^1$ be an open subset
    and $\frak{a}\in \scr{I}(U)$. 
    For a sheaf of difference module $\scr{N}$
    over 
    $(\scr{A}^{\leqslant 0 }_{|U}, 
    \widetilde{\phi}_t)$, 
    we can define the
    complexes  
    $\DR_{\leqslant \frak{a}}(\scr{N})$
    and $\DR_{< \frak{a}}(\scr{N})$
    in a similar way.
    These complexes are natural in the sense
    that a morphism
    $\xi\colon \scr{N}\to \scr{N}'$
    of 
    $(\scr{A}^{\leqslant 0 }_{|U}, 
    \widetilde{\phi}_t)$-modules induces 
    morphisms of the complexes
    $\DR_{< \frak{a}}(\scr{N})\to
    \DR_{< \frak{a}}(\scr{N}')$ and  
    $\DR_{\leqslant  \frak{a}}(\scr{N})\to
    \DR_{\leqslant \frak{a}}(\scr{N}')$
    in a natural way.
\end{remark}

Note that we have 
$u^n\colon \DR_{\leqslant\frak{a}}(\scr{M})
\simeqto \DR_{\leqslant\frak{a}+2\pi\i n s}(\scr{M})$
for $n\in \Z$. 

\begin{lemma}
There exists a unique 
sub sheaf $\Per(\scr{M})\subset \widetilde{\DR}{(\scr{M})}$
of $\scr{A}_\per$-modules
such that 
    $\Per(\scr{M})_{|U}=
    \sum_{\frak{a}\in \scr{I}(U)}
    \scr{H}^0\DR_{\leqslant \frak{a}}(\scr{M})$
for any open subset $U\subsetneq S^1$. 
\end{lemma}
\begin{proof}
If two open subsets $U$ and $V$ with  
$\emptyset \neq V\subset  U\subsetneq S^1_t $
are connected,
the restriction map 
$\scr{I}(U)\to \scr{I}{(V)}$
is an isomorphism.  
Then the restriction map 
$\sum_{\frak{a}\in \scr{I}(U)}
    \DR_{\leqslant \frak{a}}(\scr{M})
    \to \sum_{\frak{b}\in \scr{I}(V)}
    \DR_{\leqslant \frak{b}}(\scr{M})
$
is isomorphism. 
The lemma follows from this fact. 
\end{proof}
\begin{remark}\label{Periodic}
    We do not have 
    $\Per(\scr{M})= \scr{H}^0\widetilde{\DR}(\scr{M})$
    in general.
    In other words, the filtration 
    $\scr{H}^0\DR_{\leqslant \bullet}(\scr{M})$
    on $\scr{H}^0\widetilde{\DR}(\scr{M})$
    is not exhaustive in general. 
    This is one of the differences between
    our setting and the case 
    of meromorphic connections. 
    We also use the notation 
    $\Per$ for difference modules 
    over 
    $(\scr{A}_{|U}^{\leqslant 0},
    \widetilde{\phi}_t )$ in a similar way. 
\end{remark}

We set
\begin{align*}
  \scr{H}^0\DR_{\bullet }(\scr{M})
  \coloneqq 
  \{\scr{H}^0\DR_{\leqslant\frak{a}}(\scr{M})
  \mid U \subset S^1_t\colon \text{open subset, }
  \frak{a}\in \scr{I}(U)\}. 
\end{align*}
It is easy to see that 
$\scr{H}^0\DR_{\bullet}(\scr{M})$
is a pre-Stokes filtration on $\DR(\scr{M})$. 
For a morphism $\xi\colon\scr{M}\to \scr{M}'$
of analytic difference modules, 
we naturally obtain a morphism 
$\RH(\xi)\colon (\Per(\scr{M}),\scr{H}^0\DR_\bullet(\scr{M}))
\to (\Per(\scr{M}'),\scr{H}^0\DR_\bullet (\scr{M}'))$
of pre-Stokes filtered $\scr{A}_{\per}$-modules.
\begin{definition}
We define a functor 
$\RH\colon \Diffc^\mild \to\preSt$ 
by
\begin{align*}
    \RH(\scr{M})\coloneqq
(\Per(\scr{M}),\scr{H}^0\DR_\bullet(\scr{M}))
\end{align*}
for an object $\scr{M}\in {\Diffc}^\mild$
and $\RH(\xi)$ for a morphism $\xi$ in 
$\Diffc^\mild$. 
\end{definition}
\subsubsection{An example}
Let $\scr{O}_{t}(*0)$ denote 
the difference module
$(\scr{O}_{t}(*0),\phi_t)$
over itself. 
\begin{theorem}\label{O}
$\Per(\scr{O}_{t}(*0))=\scr{A}_{\per}$. 
\end{theorem}
\begin{proof}
By definition,
we have 
$\DR_{\leqslant 2\pi\i n s}(\scr{O}_{t}(*0))
=u^n\scr{A}_{\per}^{\leqslant 0}$
for every $n\in \Z$. 
It follows that 
$\scr{A}_{\per}\subset \Per(\scr{O}_{t}(*0))$.
It then remains to prove the relation
\begin{align}\label{Key}
    \widetilde{\scr{O}}_{\per|U}\cap 
    \scr{A}^{\leqslant \frak{a}}
    \subset \scr{A}_{\per|U}
\end{align}
for any $\frak{a}\in \scr{I}(U)$,
where $U\subset  S^1$ 
is an open interval. 
To see the relation \eqref{Key},
we use the presentation
    $\frak{a}(t)=
    \sum_{\ell=1}^m c_\ell\tau^{-\ell}$
as in \eqref{q(tau)}. 
Here, we fix a branch of 
$\log t$ on  $U$ 
and put $\tau\coloneqq \exp(m^{-1}\log t)$
for a positive integer $m$.

The relation \eqref{Key} is shown as follows.
Take an integer $\ell_{\max}\coloneqq \max
\{\ell\mid c_\ell\neq 0\}$.
If $\ell_{\max}\neq m$, 
then using the open subset
\begin{align*}
    U(\frak{a})\coloneqq 
    \{e^{\i\theta}\in U\mid
    \Re( \exp(-\i\theta \ell_{\max}/m)c_{\ell_{\max}})>0\},
\end{align*}
where $\theta=\Im (\log t)$,
we have
\begin{align*}
    (\widetilde{\scr{O}}_{\per}\cap 
    \scr{A}^{\leqslant \frak{a}})
    (V)=\begin{cases}
    \scr{A}_{\per}^{\leqslant 0}(V) 
    &(V\subset U(\frak{a}))\\
    \scr{A}_{\per}^{< 0}(V)
    &(V\not\subset U(\frak{a})).
    \end{cases}
\end{align*}
This implies \eqref{Key} in this case. 
The case 
$\ell_{\max}=m$ and $c_m\in 2\pi\i\Z$
is reduced to the case above. 
Assume that 
$\ell_{\max}=m$ and 
$c_m\notin 2\pi\i\Z $. 
In this case, 
for each point 
$e^{\i\theta}\in U\setminus
\{e^0,e^{\pi\i}\}$,
there exists
an integer $n$
such that 
\begin{align*}
    (\widetilde{\scr{O}}_{\per|U}\cap 
    \scr{A}^{\leqslant \frak{a}})_{e^{\i\theta}}
    \subset 
    u^n\scr{A}_{\per,e^{\i\theta}}
    ^{\leqslant 0}
    \subset \scr{A}_{\per, e^{\i\theta}}, 
\end{align*}
which implies \eqref{Key} at $e^{\i\theta}$. 
Since $c_m\notin 2\pi\i \Z$,
we have 
$ (\widetilde{\scr{O}}_{\per|U}\cap 
    \scr{A}^{\leqslant
    \frak{a}})_{\pm 1}=0$.
 \end{proof}

\subsubsection{Elementary examples}\label{EE}
Let $U\subset S^1$ be an open subset. 
For each $\frak{a}\in \scr{I}(U)$ and 
$G\in \End(\C^r)$
we consider a 
$(\scr{A}^{\leqslant 0}_{|U},
\widetilde{\phi}_t)$-module
$\scr{E}^\frak{a}\otimes \scr{R}_{G|U}$,
where we have set 
$\scr{E}^{\frak{a}}\coloneqq 
(\scr{A}^{\leqslant 0}_{|U},
\exp(\widetilde{\phi}_t(\frak{a})-\frak{a})\widetilde{\phi}_t)$
similarly as in \S \ref{ExF}. 
 Then we obtain the following:
\begin{corollary}\label{EM}
    $\Per(\scr{E}^{\frak{a}}
\otimes \scr{R}_{G|U}))
= 
\exp(-\frak{a})t^{-G}\mathscr{A}^{\oplus r}_{\per |U}$. 
\end{corollary}
\begin{proof}
The claim follows from 
Theorem \ref{O}
and the following isomorphism 
of $(\scr{A}^{\leqslant 0}_{|U},\widetilde{\phi}_t)$
-modules:
\begin{align*}
(\scr{A}
^{\leqslant \frak{b}+\frak{a}}
)^{\oplus r}
\simeqto 
\scr{A}^{\leqslant \frak{b}}
\otimes \scr{M},
\quad 
\bm{f}\mapsto \exp(-\frak{a})t^{-G}\bm{f}
\end{align*}
where we put 
$\bm{f}={}^t(f_1,\dots,f_r)$,
$\scr{M}=\scr{E}^{\frak{a}}
\otimes \scr{R}_{G|U}$, and 
$\frak{b}\in\scr{I}(U)$. 
\end{proof}
\begin{proof}[A proof of Corollary \ref{Cor}]
The fact 
that Theorem \ref{Vanishing}
implies 
Corollary \ref{Cor} easily follows from the 
isomorphism
\begin{align*}
    \DR_{\leqslant \frak{a}}(\scr{M})
    \simeqto 
    \DR_{\leqslant 0} (\scr{E}^{\frak{a}}\otimes \scr{M}_{|U}) 
\end{align*}
given by the multiplication of 
$\exp(-\frak{a})$. 
\end{proof}

\subsection{Existence of local splittings}
\label{RH-reg}
In this subsection,
we 
 prove the following:
\begin{theorem}\label{ThR}
If $\scr{M}\in \Diffc^\mild$, 
then $\RH(\scr{M})$ is a
Stokes filtered $\scr{A}_\per$-module. 
\end{theorem}

\subsubsection{Asymptotic expansions}
As a preliminary,
we recall the basic theory of asymptotic expansions. 
Let $\cal{A}$ be the subsheaf of 
$\widetilde{\scr{O}}$ whose sections have the 
asymptotic expansion in $\widehat{\scr{O}}(*0)$. 
By the Borel-Ritt theorem,
we have an exact sequence 
\begin{align}\label{BR}
    0\longrightarrow \scr{A}^{<0}
    \longrightarrow \cal{A}
    \xrightarrow{\asy}\varpi^{-1}\widehat{\scr{O}}{(*0)}
    {\longrightarrow} 0
\end{align}
where the symbol `$\asy$' denotes the asymptotic expansion. 
We note that the relation $\cal{A}\subset \scr{A}^{\leqslant 0}$
holds
as subsheaves of $\widetilde{\scr{O}}$.

Take a positive integer $m$
and let $\rho\colon S^1_\tau\to S^1_t$ 
be a finite covering 
as in \S \ref{Ram}. 
Take a local section 
$\frak{s}\colon U\to S^1_\tau$
of $\rho$ (i.e. fix a branch of $\tau=t^{1/m}$). 
Then we have an
exact sequence 
\begin{align}\label{pre-m}
    0\longrightarrow \scr{A}^{<0}_{t|U}
    \longrightarrow \cal{A}_{|U}^{(m)}
    \xrightarrow{\asy_{\frak{s}}}
    \widehat{\scr{O}}_\tau{(*0)}_{S^1|U}
    {\longrightarrow} 0
\end{align}
by the pull back of \eqref{BR},
where we put 
$\cal{A}^{(m)}\coloneqq
\rho_*\cal{A}\cap \scr{A}^{\leqslant 0}\subset 
\rho_*\widetilde{\scr{O}}_\tau $,
and 
$\asy_{\frak{s}}=\frak{s}^{-1}\asy$.

\subsubsection{Key fact}
Let $\scr{M}$
be a mild difference module. 
By Lemma \ref{Graded},
we have a graded module $\scr{N}$
of $\scr{M}$. 
Let $m$
be a positive integer and use the notations
in \S \ref{Ram} and \S \ref{Extended}. 
Assume that there exists an
isomorphism 
$\widehat{\Xi}\colon
\rho^*\widehat{\scr{M}}\simeqto
\rho^*\widehat{\scr{N}}$
of difference modules 
over 
$(\widehat{\scr{O}}_{\tau}(*0),\widehat{\phi}_\tau)$. 

\begin{theorem}
[{\cites{BF,Galois}}]\label{CReg}
Let $\scr{M},\scr{N},m,$ and $\widehat{\Xi}$
be as above. 
Then, for any point $e^{\i\theta}\in S^1_t$, 
there exist an open neighborhood $U$
of $e^{\i\theta}$, 
an isomorphism 
\begin{align*}
    \Xi_{U} &\colon 
(\cal{A}^{(m)}{\otimes}
\scr{M}_{S^1})_{|U}
\simeqto 
(\cal{A}^{(m)}{\otimes}
\scr{N}_{S^1})_{|U},
\end{align*}
and a section $\frak{s}:U\to S^1_\tau$ such that 
the formal completion
\begin{align*}
    \id_{\widehat{\scr{O}}_\tau(*0)}
    \otimes {\Xi}_{U}
    \colon   
    \rho^*\widehat{\scr{M}}
    \longrightarrow 
    \rho^*\widehat{\scr{N}}
\end{align*}
via $\asy_{\frak{s}}$
 coincides 
with $\widehat{\Xi}$,
where we have used the
identifications 
\begin{align*}
\widehat{\scr{O}}_\tau(*0)\otimes
    (\cal{A}^{(m)}{\otimes}
    \scr{M}_{S^1})_{|U}=\rho^*\widehat{\scr{M}},
    \text{ and }
\widehat{\scr{O}}_\tau(*0)\otimes
    (\cal{A}^{(m)}{\otimes}
    \scr{N}_{S^1})_{|U}=\rho^*\widehat{\scr{N}}. 
\end{align*}
\end{theorem}
\begin{proof}
Take a basis 
$e_1,\dots,e_r$ of $\scr{M}$ over
$\scr{O}_t(*0)$. 
Take a basis $f_1,\dots,f_r$
of $\scr{N}$ over $\scr{O}_t(*0)$.
Let $\psi_{\scr{M}}$ and
$\psi_{\scr{N}}$ be the automorphisms on 
$\scr{M}$ and $\scr{N}$, respectively. 
Then we have
$\psi_{\scr{M}}
\bm{e}=
A(t) \bm{e}$
and 
$\psi_{\scr{N}}
\bm{f}=
B(t) \bm{f}$
where we put
$\bm{e}={}^t(e_1\ \cdots\ e_r)$
and $\bm{f}={}^t(f_1\ \cdots\ f_r)$. 
By assumption, there exists a invertible matrix 
$\widehat{Y}(\tau)
\in \GL_r(\widehat{\scr{O}}_\tau(*0))$
which
satisfies the difference equation
$\widehat{\phi}_{\tau}(\widehat{Y})({\tau})
=B(\tau^m)\widehat{Y}(\tau )A(\tau^m)^{-1}$. 
By \cite[Theorem 9.1, Theorem 11.1]{Galois}, 
which is essentially due to
Braaksma-Faber \cite[Theorem 4.1]{BF}
(or, more generally, 
\cite[Theorem 11.7, Remarks 11.9, Theorem 11.10,
Remark 11.11]{Galois},
which uses the result of Immink \cites{Immink, Immink2})
for any point 
$e^{\i\theta}\in S^1$
for sufficiently small 
$U$ and a choice
$\frak{s}\colon U\to S^1_\tau$
of the branch of $\tau=t^{1/m}$,
there exists a 
section
$Y(t)\in \GL_r(
\cal{A}^{(m)} (U))$
such that 
$\asy_{\frak{s}}(Y)=\widehat{Y}$ and 
$\widetilde{\phi}_{t}({Y})
(t)=
B(t)
{Y}(t)
A(t)^{-1}.$
Set $\Xi_{U}\bm{e}=Y(t)\bm{f}$. 
Then it follows that $\Xi_{U}$
is a morphism that satisfies the 
conditions.  
\end{proof}
\subsubsection{A proof of Theorem \ref{ThR}}
Let $\scr{M}$ be a 
mild difference module.
Let $\scr{N}$ be the graded module of $\scr{M}$. 
 
Then, by Theorem \ref{CReg}, for each point 
$e\in S^1_t$,
there is an open neighborhood 
$U$ of $e$
and 
an isomorphism
\begin{align*}
    \Xi_U\colon \scr{M}_{|U}\simeqto 
    \scr{N}_{|U}
\end{align*}
of difference modules over 
$(\scr{A}^{\leqslant 0}_{|U}, \widetilde{\phi}_t)$. 
We obtain the theorem
by Corollary \ref{EM}.
\qed

\subsection{Statement of the main theorem} 
We are now in the position 
to state the main result of 
this paper:

\begin{theorem}\label{Main}
The induced functor 
\begin{align*}
    \RH\colon \Diffc^\mild \to \Stokes,
    \quad
    \scr{M}\mapsto \RH(\scr{M})=
    (\Per(\scr{M}),\scr{H}^0 \DR_\bullet (\scr{M}))
\end{align*}
is an equivalence of categories. 
\end{theorem} 
We give a proof of the fact that 
$\RH$ is fully faithful in \S \ref{FF}.
Then we use the fact to prove that 
$\RH$ is essentially surjective 
in \S \ref{ES}.
\subsection{Fully faithfulness}\label{FF}
We shall prove that the functor $\RH$ 
is fully faithful. 
\subsubsection{}
Let $\scr{M}=(\scr{M},\psi)$ and 
$\scr{M}'=(\scr{M},\psi')$ be  difference modules
over $(\scr{O}_t(*0),\phi_t)$. 
The 
module 
$\homscr_{\scr{O}}(\mathscr{M},\scr{M}')$
of morphisms of $\mathscr{O}_t(*0)$-modules
from $\scr{M}$ to $\scr{M}'$
is equipped with the automorphism 
$\psi_{\scr{M},\scr{M}'}(h)\coloneqq  
\psi'\circ h\circ \psi^{-1}$
for $h\in \homscr_{\scr{O}}(\mathscr{M},\scr{M}')$.
For simplicity of the notations, 
we set
$(\scr{L},\scr{L}_{\bullet})\coloneqq
(\Per(\scr{M}),\scr{H}^0\DR_\bullet(\scr{M}{}))$
and 
$(\scr{L}',\scr{L}'_\bullet)
\coloneqq
(\Per(\scr{M}'),\scr{H}^0\DR_\bullet(\scr{M}')).
$
\begin{lemma}\label{RH_0}
There is a natural morphism
\begin{align*}
    \RH_{\leqslant 0}\colon 
    \scr{H}^0\DR_{\leqslant 0}(\homscr_{\scr{O}}
    (\scr{M},\scr{M}'))\longrightarrow
    \homscr(\scr{L}, \scr{L}')_{\leqslant0}.
\end{align*}
\end{lemma}
\begin{proof}
A local section on an open subset $U$ of  
$\scr{H}^0\DR_{\leqslant 0}(\homscr_{\scr{O}}(\scr{M},\scr{M}'))$
is regarded as a morphism
\begin{align*}
    \xi\colon \scr{M}_{|U}
    \to \scr{M}_{|U}
\end{align*}
which is compatible with the difference operators. 
For each $\frak{a}\in\scr{I}(U)$,
it also sends 
$\scr{A}^{\leqslant \frak{a}}
\otimes \scr{M}_{|U}$
into 
$\scr{A}^{\leqslant \frak{a}}
\otimes\scr{M}'_{|U}$.
It follows that 
$\xi$ sends $\scr{L}_{\leqslant \frak{a}}$
into $\scr{L}'_{\leqslant\frak{a}}$.
Hence it defines
a local section 
$\RH_{\leqslant 0}(\xi)
\in \homscr(\scr{L},\scr{L'})_{\leqslant 0}$. 
\end{proof}
\subsubsection{}\label{RH^tau}
For each point 
$e\in S^1_t$, there is an open neighborhood 
$U$ of $e$ 
such that 
$\RH_{\leqslant 0}$
is isomorphic to 
the direct sum of the morphisms of the form
\begin{align*}
    &\scr{H}^0\DR_{\leqslant 0}
    (\homscr
    (\scr{E}^{\frak{a}}\otimes\scr{R}_{G|U},
    \scr{E}^{\frak{b}}\otimes \scr{R}_{H|U}))\\
    &\longrightarrow 
    \homscr(\exp(-\frak{a})
    t^{-G}\scr{A}^{\oplus r}_{\per|U}, 
    \exp(-\frak{b})
    t^{-H}\scr{A}^{\oplus r'}_{\per|U})
    _{\leqslant 0}
\end{align*}
where $\frak{a},\frak{b}\in \scr{I}(U)$,
$G\in \End(\C^r)$, and 
$H\in \End(\C^{r'})$ by Theorem \ref{CReg}.  
It then follows that 
the morphism $\RH_{\leqslant 0}$
is 
an isomorphism. 
\subsubsection{}
Let us set
$\scr{N}=\homscr_{\scr{O}}(\scr{M},\scr{M}')$
and we define $\DR(\scr{N})$ by the complex
$\scr{N}\xrightarrow{\psi-\id} \scr{N}$
in degree zero and one, where $\psi$ denotes
the automorphism on $\scr{N}$ naturally induced
from $\scr{M}$ and $\scr{M}'$. 
Then we have $\scr{H}^0\DR(\scr{N})=
\Hom(\scr{M},\scr{M}')$. 
By the projection formula,
we have
$\DR(\scr{N})=\R\varpi_*\DR_{\leqslant 0}(\scr{N})$. 
By Theorem \ref{Vanishing},
we obtain 
$\R\varpi_*\DR_{\leqslant 0}(\scr{N})
=\R\varpi_*\scr{H}^0\DR_{\leqslant 0}(\scr{N})$
and hence 
\begin{align*}
    \Hom(\scr{M},\scr{M}')=
    \varpi_*\scr{H}^0\DR_{\leqslant 0}
    (\homscr_{\scr{O}}(\scr{M},\scr{M}')).
\end{align*}
Then, the morphism
$\RH=\varpi_*\RH_{\leqslant 0}$
gives an isomorphism between 
$\Hom(\scr{M},\scr{M}')$
and 
$\varpi_*\homscr(\scr{L},\scr{L}')_{\leqslant 0}
=\Hom_{\Stokes}(\scr{L},\scr{L}')$.
\subsection{Essential surjectivity}\label{ES}
To complete the proof of 
Theorem \ref{Main}, 
we shall prove the essential surjectivity of 
the functor $\RH$. 
\subsubsection{Graded case}
We shall firstly show the 
essential surjectivity of
$\RH$ in the graded case:

\begin{lemma}\label{ESG}
Let $\scr{G}$ be a graded 
Stokes filtered $\mathscr{A}_\per$-module. 
    There exists 
    a mild graded difference module 
    $\scr{N}$ over $(\scr{O}_t(*0),\phi_t)$
    such that  
    $\RH(\scr{N})\simeq \scr{G}$. 
\end{lemma}
\begin{proof}
Take a positive integer 
$m$ such that $\Sigma(\scr{G})\subset \scr{I}_m$. 
Then 
$\rho^{-1}\Sigma(\scr{G})\subset 
\rho^{-1}\scr{I}_m=\scr{I}_{m,\tau}$
is a trivial sheaf of finitely many
$ (2\pi\i \Z s)$-orbits, 
which admits a $\mu_m$-action.
Take a section $\frak{a}_i$ for each $i=1,\dots,\ell$
in the orbit, so that 
\[\rho^{-1}_m\Sigma(\scr{G})
=\bigsqcup_{i=1}^\ell (\frak{a}_i+2\pi\i s \Z_{S^1}).\] 
For each $i$, 
let $\gr_{\frak{a}_{i}}(\rho^{-1}\scr{G})$ be
a local system of 
$\C$-vector space 
defined as follows:
\begin{align*}
    \gr_{\frak{a}_{i}}(\rho^{-1}\scr{G})(U)=
    \gr_{\frak{a}_{i,U}}\scr{G}(\rho(U))
\end{align*}
where $U$ is an open subset 
such that $\rho_{|U}$ is a homeomorphism
and $\frak{a}_{i,U}\in \scr{I}_m(U)$
is a section such that 
$\rho_{|U}^*\frak{a}_{i,U}=\frak{a}_{i|U}$.
Then, there is a matrix $G_i$
such that the monodromy of $\gr_{\frak{a}_i}(\scr{G})$
is given by $\exp(2\pi \i m G_i)$ for some basis at 
a fiber. 
Set 
$\scr{N}'\coloneqq 
\bigoplus_{i}\scr{E}^{\frak{a}_i}\otimes \rho^*\scr{R}_{G_i}$.
Then, by the $\mu_m$-invariant construction given above, 
there is a 
mild graded difference module 
$\scr{N}$ such that $\rho^*\scr{N}\simeq \scr{N}'$. 
We can easily check that $\RH(\scr{N})\simeq \scr{G}$. 
\end{proof}

\subsubsection{Classification theorem}\label{MSthm}
For a difference module $\scr{N}$ over 
$(\scr{O}_t(*0),\phi_t)$, we set 
\begin{align*}
    \aut^{<0}(\scr{N})
    =\id_{\scr{N}}+
    \scr{H}^0\DR_{<0}(\scr{E}nd(\scr{N})).
\end{align*}
The following Malgrange-Sibuya type classification theorem
for difference modules plays a key role 
in the proof of essential surjectivity of $\RH$:
\begin{theorem}\label{ClMe}
Let $\scr{N}$ be a mild graded
difference module over 
$(\scr{O}_{t}(*0),\phi_t)$. 
Let 
$m$ be a positive integer 
such that $\rho^*_m\scr{N}$ is 
mild elementary. 
Then there is a natural one-to-one correspondence 
between the set
\begin{align*}
    H^1(S^1 ,\aut^{<0}(\scr{N}))
\end{align*}
and 
the set of isomorphism classes of  
pairs
$(\scr{M},\widehat{\Xi})$ of 
\begin{itemize}
    \item a mild difference module 
    $\scr{M}$ over $(\scr{O}_t(*0),\phi_t)$,
    and 
    \item an isomorphism 
    $\widehat{\Xi}\colon \rho^*\widehat{\scr{M}}
    \simeqto \rho^*\widehat{\scr{N}}$
    of difference modules over 
    $(\widehat{\scr{O}}_\tau(*0),\widehat{\phi}_\tau)$.
\end{itemize}
\end{theorem}
\begin{proof}
This theorem is essentially proved in 
\cite[Theorem 11.12]{Galois}.
We shall recall the construction of 
the pair
$(\scr{N},\widehat{\Xi})$ from the cohomology 
class $[g]$
for the convenience of the reader. 
Take a finite open covering 
$S^1=\bigcup_{\alpha\in \Lambda} U_\alpha$ 
such that 
\begin{itemize}
\item 
$\Lambda=\Z/N\Z$ for a positive integer $N$
such that, 
\item
$U_\alpha$ is an open interval for each 
$\alpha\in \Lambda$, and
\item
$U_\alpha\cap U_\beta=\emptyset$ for 
$\beta\notin \{\alpha-1,\alpha,\alpha+1\}$,
\end{itemize}
and a representative 
$g=(g_\alpha)_\alpha\in
\prod_{\alpha\in \Lambda}
H^0(V_{\alpha},
\aut^{<0}(\scr{N}))$,
where
$V_{\alpha}\coloneqq U_{\alpha}\cap U_{\alpha+1}$.
Fix a frame 
$\scr{N}= 
\bigoplus_{i=1}^r \scr{O}_t(*0)e_i$, 
then, for each $\alpha$, 
the section $g_\alpha$ is
identified with a section $(g_{\alpha,ij})\in
\GL_r(\cal{A}^{(m)}(V_{\alpha}))$
via $g_\alpha(e_i)=\sum_jg_{\alpha,ij}e_j$. 
Then we obtain 
the map 
\[H^1(S^1, \aut^{<0}(\scr{N}))
\to H^1(S^1,\GL_r(\cal{A}^{(m)})).\]
By the Malgrange-Sibuya theorem,
this map is trivial. 
It implies that 
there exists a family 
$(h_\alpha)_\alpha\in 
\prod_{\alpha\in\Lambda} 
H^0(U_\alpha, \GL_r(\cal{A}^{(m)}))$
such that 
$g_\alpha=h_{\alpha}^{-1} h_{\alpha+1}$
on $V_{\alpha}$ 
(taking finer covering if necessary). 
Let $\psi$ be 
the difference operator on $\scr{N}$. 
Using the frame,  we have 
$\psi=A^0(t)\phi_t^{\oplus r} $
for some $A^0(t)\in \GL_r(\scr{O}_t(*0))$.
Note that we have the equality
$ \widetilde{\phi}_t(g_\alpha)
=A^{0}g_\alpha(A^0)^{-1}$.
We then set 
\[\widetilde{A}_\alpha (t)
=\widetilde{\phi}_t(h_{\alpha})
A^0(t)h_\alpha^{-1}
\in \GL_r(\cal{A}^{(m)}(U_\alpha))\]
for each $\alpha\in \Lambda$. 
Since we have
\begin{align*}
    \widetilde{A}_{\alpha}\widetilde{A}_{\alpha+1}^{-1}
    =\widetilde{\phi}_\tau(h_{\alpha})
A^0(t)h_\alpha^{-1}
h_{\alpha+1}
A^0(t)^{-1}\widetilde{\phi}_\tau(h_{\alpha+1}^{-1})
=\id
\end{align*}
on $V_{\alpha}$, 
there exists a unique $A(t)\in \GL_r(\scr{O}_\tau(*0))$
such that $A(t)_{|U_\alpha}=\widetilde{A}_\alpha$. 
Then $\psi\coloneqq A(t)\phi_t^{\oplus r}$ on 
$\scr{M}\coloneqq
\scr{N}
=\bigoplus_{i=1}^r \scr{O}_\tau(*0)e_i$
defines a new difference module 
$(\scr{N},\psi)$ with the trivial isomorphism
$\widehat{\Xi}$ between the formal completions. 
\end{proof}

\subsubsection{End of the proof of Theorem \ref{Main}
$($Essential surjectivity$)$}
Let $(\scr{L},\scr{L}_{\bullet})$ be an object 
in $\Stokes$. 
There exists a positive integer 
$m$ such that 
$ \Sigma(\scr{L})\subset \scr{I}_m$. 
By Theorem \ref{ClSt}, 
the pair $(\scr{L},\scr{L}_{\bullet})$
corresponds to a cohomology class
$[\scr{L},\scr{L}_{\bullet}]$
in 
$H^1(S^1,
\aut^{<0}(\scr{G}))$
with $\scr{G}=\scr{A}_{\per}\otimes \gr\scr{L}$. 
There exists a
mild difference module 
$\scr{N}$ over $(\scr{O}_t(*0),\phi_t)$
such that $\RH(\scr{N})=\scr{G}$ by Lemma \ref{ESG}.
By the fully-faithfulness of $\RH$, 
we have an isomorphism 
\begin{align*}
    H^1(S^1,\aut^{<0}(\scr{N}))
    \simeq 
    H^1(S^1,
    \aut^{<0}(\scr{G})).
\end{align*}
Let $[\scr{M}]$ be the class in 
$H^1(S^1, \aut^{<0}(\scr{N}))$
which corresponds to
$[\scr{L},\scr{L}_{\bullet}]$
by the above isomorphism. 
By Theorem \ref{ClMe},
there exists a difference module 
$\scr{M}$ over $(\scr{O}_t(*0),\phi_t)$
which corresponds to the class $[\scr{M}]$. 
Then, by the construction,
we have 
$\RH(\scr{M})\simeq 
(\scr{L},\scr{L}_{\bullet})$,
which implies the essential surjectivity  of $\RH$.
\qed

\subsection{Rank one non-trivial examples}\label{R1}
\subsubsection{A regular singular difference module}
For a complex number $\alpha\in\C\setminus \Z$, 
we consider a 
difference module 
$\scr{B}_{\alpha}\coloneqq (\scr{O}_t(*0),\psi_\alpha )$
defined as 
$\psi_\alpha=(1+\alpha t)\phi_t$.
It is easy to see that 
the formal completion 
$\widehat{\scr{B}}_{\alpha}$
is isomorphic to 
$\widehat{\scr{R}}_{\alpha}$,
where $\alpha$ is regarded as a $(1\times 1)$-matrix.
\begin{proposition}\label{REX}
We have the following description of
$\Per({\scr{B}}_{\alpha}):$
\begin{align*}
    \Per({\scr{B}}_{\alpha})_{|U}=
    \begin{cases}
    \scr{A}_{\per|U} \Gamma(s)/\Gamma(s+\alpha)
    &(e^{\pi\i}\notin U)
    \\
    \scr{A}_{\per|U} 
    (1-u)\Gamma(s)
    /
    (1-e^{2\pi \i\alpha }u)\Gamma(s+\alpha)
    &
    (e^{0}\notin U), 
    \end{cases}
\end{align*}
where $U$ is an open subset in $S^1$, 
the symbol $\Gamma(s)$ denotes the Gamma function,
and we set $s=t^{-1}$
and $u=\exp(2\pi\i s)$ as in $\S \ref{Another}$. 
\end{proposition}
\begin{proof}
By the relation $\Gamma(s+1)=s\Gamma(s)$, the equality
\[\psi_\alpha (\Gamma(s)/\Gamma(s+\alpha))
=\Gamma(s)/\Gamma(s+\alpha)\] 
holds, 
which implies that 
$\Gamma(s)/\Gamma(s+\alpha)\in
\widetilde{\DR}({\scr{B}_{\alpha}})(U_+)$ 
for $U_+=S^1\setminus \{e^{\pi\i}\}$. 
By the Stirling formula,
we moreover obtain that 
$\Gamma(s)/\Gamma(s+\alpha)\in 
\DR_{\leqslant0}(\scr{B}_{\alpha})(U_+)$,
whose germ at each point $x\in U_+$ is not in 
$\DR_{<0}(\scr{B}_{\alpha})_x$. 
By Theorem \ref{O},
we obtain the first half of the proposition.
The latter half can be proved in a similar way
by using the reflection formula:
$(1-u)\Gamma(s)=(-2\pi\i)e^{\pi\i s}/\Gamma(1-s)$. 
\end{proof}
\subsubsection{A mild difference module}
Let 
$\scr{E}_\Gamma\coloneqq 
(\scr{O}_t(*0),\psi_\Gamma)$
be a difference module with
$\psi_\Gamma=
\exp(\frak{l}(s+1)
-\frak{l}(s))t\phi_t$ where 
$\frak{l}(s)=s\log s$.
It is not difficult to see
that the formal completion 
$\widehat{\scr{E}}_\Gamma$
is isomorphic to 
$\scr{E}^{\frak{a}_\Gamma}\otimes 
\scr{R}_{-1/2}$
where $\frak{a}_\Gamma(s)=-s$. 

\begin{proposition}
    We have the following description
    of $\Per (\scr{E}_\Gamma)$$:$
    \begin{align*}
        \Per(\scr{E}_\Gamma)_{|U}
        =\begin{cases}
            \mathscr{A}_{\per|U}s^{-s}\Gamma(s) &(e^{\i\pi}\notin U)\\
            \scr{A}_{\per|U}(1-u)s^{-s}\Gamma(s)&
            (e^0\notin U)
        \end{cases}
    \end{align*}
    where the notations except $s^{-s}$
    are defined as in Proposition \ref{REX}
    and we define 
    $s^{-s}=\exp(-s\log s)$.
    Note that  
    the choice of the branch of $log s$
    does not matter in the definition of the modules. 
\end{proposition}
\begin{proof}
The proof is parallel to that for
Proposition \ref{REX} and left to the reader. 
\end{proof}

\section{Proof of  Theorem \ref{Vanishing}}\label{PRH}
In this section, 
we give a proof of Theorem \ref{Vanishing}
using the method in \cite{Immink}.  
\subsection{Preliminary}
\subsubsection{Reduction}
By Theorem \ref{CReg},
we may assume that 
for any $e^{\i\theta}\in S^1$,
there exists an open neighborhood
$U$ of $e^{\i\theta}$
such that the restriction of the 
difference module
$\scr{M}$ in the statement of the 
Theorem \ref{Vanishing}
to $U$
is
isomorphic to the direct sum of 
the difference module of the form
$\scr{E}^{\frak{a}}\otimes \scr{R}_{G|U}$,
where notations are as in \S \ref{EE}. 
We may also assume 
that $G$
is of the form
$G=\gamma \id_{\C^r}+N$
where $N$ is a nilpotent Jordan matrix
and $\gamma \in \C$.
Hence it reduces to prove the following claim:
\begin{claim}\label{Claim}
    Assume that 
    $ \scr{M}_{|U}=
\scr{E}^{\frak{a}}\otimes \scr{R}_{G|U}$
for $\frak{a}\in \scr{I}(U)$
and $G=\gamma\id_{\C^r}+N$ with 
$\gamma\in\C$ and a nilpotent 
matrix $N\in \End(\C^r)$. 
Then, for any $\bm{f}\in 
\scr{M}_{|U,e^{\i\theta}}$
there exists
$\Lambda_\theta(\bm{f}) \in \scr{M}_{|U,e^{\i\theta}}$
such that  
$\nabla_{\widetilde{\psi}}
\Lambda_\theta(\bm{f})=\bm{f}$.
Moreover,
if $\bm{f}\in (\scr{A}^{< 0}\otimes \scr{M}_{S^1})
_{e^{\i\theta}}$, 
then we have $\Lambda_\theta(\bm{f})\in
(\scr{A}^{< 0}\otimes \scr{M}_{S^1})_{e^{\i\theta}}$. 
\end{claim}
The remaining part of the theorem 
is easy and left to the 
reader.  
We shall prove this claim
in the case
$0<\theta<{\pi}$.
The proof in the case $\pi<\theta<2\pi$ 
is similar to that in 
the case above.
The proof in the case $\theta\in\{0,\pi\}$
is essentially given in 
\cite{Immink}*{\S 9.3}.

\subsubsection{Notations}
We shall fix some notations. 
Let $c$ be a point in $\C$. 
Let $\theta$ be real number
and $\varepsilon$ be a positive number. 
Then we set
\begin{align*}
    S_c(\theta;\varepsilon)=\{s\in\C\mid 
    s=c+Re^{\i\vartheta} \text{ for }
    |\theta-\vartheta|<\varepsilon
    \text{ and } R>0\}. 
\end{align*}
For $\bm{f}\in 
\scr{M}_{|U,e^{\i\theta}}$,
using the notation \S \ref{Another},
we take a representative
$f={}^t(f_1,\dots,f_r)$ of $\bm{f}$
defined on 
$S_c(\theta;\varepsilon)$
for sufficiently small $\varepsilon>0$ 
and a suitable point $c\in\C$.
Without loss of generality,
we may assume that $|c|>1$ and that
there exist 
 positive constants $C$ and 
$N$ such that 
\begin{align}\label{Moderate}
    |f(s)|\leq  C |s|^N
\end{align}
for $s\in S_c(\theta;\varepsilon)$,
where we put 
$|f(s)|=\sqrt{\sum_{i=1}^r |f_i(s)|^2}$.
In the case where we have
$\bm{f}\in (\scr{A}^{< 0}\otimes \scr{M}_{S^1})
_{e^{\i\theta}}$,
for any 
positive constant $N'$, 
there exist a point $c'_{N'}\in S_c$
and a constant $C_{N'}$ such that 
we have 
\begin{align}\label{Rapid}
    |f(s)|\leq C_{N'}|s|^{-N'}
\end{align}
for $s\in S_{c_{N'}}(\theta,\varepsilon)$.

We fix the branch of 
$\log s$ on 
$S_c(\theta;\varepsilon)$ 
and set 
$s^{-1/m}\coloneqq \exp(m^{-1}\log t)$. 
Then,
$\frak{a}$
is expressed as 
\begin{align*}
    \frak{a}(s)=\sum_{k=1}^m c_k s^{k/m}
    \quad (c_k\in\C). 
\end{align*}
Taking the shift by $2\pi\i \Z s$ if necessarily, 
we may assume that $c_m$ satisfies
the inequality
\begin{align}\label{mu}
    -2\pi+{R}(c_m,\theta,\varepsilon)
    +\mathrm{Im} (c_m)<0,
\end{align}
where we set 
\begin{align*}
    R(c_m,\theta,\varepsilon)=
    \max\{0,-\mathrm{Re}(c_m)\}
    \max\{\tan^{-1}(\theta-\varepsilon),
    \tan^{-1}(\theta+\varepsilon)\}.
\end{align*}
We set
$A(s)=\exp(\frak{a}(s+1)-\frak{a}(s))
(1+s^{-1})^{-G}$.
Then $\scr{M}_{|U}$
is identified with the module
$\left((\scr{A}_{|U}^{\leqslant 0}
)^{\oplus r},
A(s)\widetilde{\phi}_t^{\oplus r} 
\right)$.
We also use the notation
$Y(s)=\exp(-\frak{a}(s))s^{G}$
with $s^G=\exp(G\log s)$,
which is also defined on 
$S_c(\theta;\varepsilon) $.

\subsubsection{Preliminary estimate}
We shall recall an estimate by Immink: 
\begin{lemma}[{\cite{Immink}*{Lemma 8.12}}]
\label{Immink-Estimate}
Let $ S\subset\mathbb{C}$
be a connected open subset 
such that $|{\sf s}|\geq 1$ for ${\sf s}\in S$. 
For $\mu\in\C$, $\nu\in \R$, 
and a holomorphic function
${\Psi}\colon S\to \C$, we set
\begin{align*}
    {\Phi}(\zeta)\coloneqq \mu\zeta
    +\nu\log|\zeta|+\Psi(\zeta),
    \quad \quad (\zeta\in S).
\end{align*}
Assume that 
we have the inequalities
\begin{align}
    &|\Psi(\zeta)|\leq C|\zeta|^{1-\epsilon}&
    (\zeta\in S)\\
    &|\Psi'(\zeta)|\leq C|\zeta|^{-\epsilon}
    &(\zeta\in S)
\end{align}
where $C$ and $\epsilon$ are positive numbers. 
For ${\sf s}\in S $, and $\alpha \in\R$, set 
\begin{align*}
    \ell(\mathsf{s};\alpha)\coloneqq 
    \{\zeta\in \C\mid \zeta= {\sf s}+Te^{\i\alpha},  0<T\}. 
\end{align*}
We also assume that 
there exists a positive constant 
$\delta$ the inequality
    $\mathrm{Re}(\mu e^{\i\alpha})\geq \delta$
    holds. 
Further, let $P$
be a polynomial with positive coefficients.
Then there exists a positive number
$K$ which is fully determined by
the constants $\delta,\nu,C$ 
and $\epsilon$, and the polynomial $P$
such that 
    \begin{align*}
        \int_\ell 
        \left|\exp(\Phi({\sf s})-\Phi(\zeta))
        P(\log|{\sf s}/\zeta|)d\zeta
        \right|\leq K.
    \end{align*}
\end{lemma}

\subsection{Definition of the Splitting}

\subsubsection{Family of paths}
Take a point  
$p\in S_c
(\theta;\varepsilon)$
and $q\in S_p
(\theta;\varepsilon)$ such that the point 
$q+1$ is in $S_p(\theta;\varepsilon)$.
For $s\in S_q(\theta;\varepsilon)$,
set $s'=s+\tfrac{1}{2}$. 
Then we define the family of 
paths
\begin{align*}
    C(s)=
    p+\R_{>0}(s'-p)
    =\{p+T(s'-p)\mid  T> 0\}
\end{align*}
for each $s\in S_q(\theta;\varepsilon)$. 
\subsubsection{Integral operator}
We consider the following integral
\begin{align}
\Lambda_\theta (f)(s)\coloneqq
-f(s)+
    Y(s)
    \int_{{{C}} (s)}
    \frac{Y(\zeta)^{-1}f(\zeta)}
    {1-e^{ 2\pi\i (s-\zeta) }}d\zeta    
\end{align}
for the family of paths $C(s)$, 
$s\in S_q(\theta;\varepsilon)$.
The proof of the following lemma 
will be given 
in the next subsection:
\begin{lemma}\label{Est}
  There exist 
  positive constants $K$ and ${N}$ such that 
  the inequality 
   \begin{align}
      \left| Y(s)
       \int_{C(s)}
       \frac{Y(\zeta)^{-1}f(\zeta)}
       {1-e^{2\pi\i(s-\zeta)}}d\zeta\right|
       \leq K|s|^{{N}}
    \end{align} 
  holds for $s\in S_{q}(\theta,\varepsilon)$.
  In particular, the integral is well defined. 
  If moreover $f$ rapid decays,
  then for any positive $N'$ 
  there are positive constants $K',\varepsilon'$
  and a point $c_{N'}\in S_{c}(\theta,\varepsilon)$
  such that the inequality
  \begin{align}\label{rapid2}
      \left| Y(s)
       \int_{C(s)}
       \frac{Y(\zeta)^{-1}f(\zeta)}
       {1-e^{2\pi\i(s-\zeta)}}d\zeta\right|
       \leq K'|s|^{-N'}
    \end{align}
holds for $s\in S_{c_{N'}}(\theta,\varepsilon')$. 

\end{lemma}
By this lemma, $\Lambda_\theta (f)(s)$
is a holomorphic function on 
$S_q(\theta;\varepsilon)$. 
We shall check that 
this lemma implies 
Claim \ref{Claim}
in the case $0<\theta<\pi$:
\begin{lemma}
We have 
$\nabla_{\widetilde{\psi}}\Lambda_\theta(f)(s)=f(s)$
on $s\in S_q(\theta;\varepsilon)
\cap S_{q+1}(\theta;\varepsilon)$. 
\end{lemma}
\begin{proof}
Note that $A(s)Y(s+1)=Y(s)$. 
Then we have 
\begin{align*}
    \nabla_{\widetilde{\psi}} \Lambda_\theta (f)
    =&-A(s)f(s+1)+A(s)Y(s+1)\int_{C(s+1)}
    \frac{Y(\zeta)^{-1}f(\zeta)}
    {1-e^{2\pi\i(s-\zeta)}}\\
    &+f(s)-Y(s)\int_{C(s)}
    \frac{Y(\zeta)^{-1}f(\zeta)}
    {1-e^{2\pi\i(s-\zeta)}}d\zeta\\
    =&-A(s)f(s+1)+f(s)+A(s)Y(s+1)\int_{C(s+1)-C(s)}\frac{Y(\zeta)^{-1}f(\zeta)}
    {1-e^{2\pi\i(s-\zeta)}}d\zeta\\
    =&-A(s)f(s+1)+f(s)+A(s)Y(s+1)Y(s+1)^{-1}f(s+1)\\
    =&
    f(s). 
\end{align*}
This completes the proof.
\end{proof}
This lemma implies the first part of 
Claim \ref{Claim}.
The last part follows from
\eqref{rapid2}.
\subsection{Estimates}
We shall give a proof of 
the first part of Lemma \ref{Est}. 
The latter part can 
be shown in a similar way
and is left to the reader.
\subsubsection{}
By the definition
of $Y(s)$, the matrix-valued 
function
$Y(s)Y(s')^{-1}$
is bounded on $S_c(\theta;\varepsilon)$,
and hence there is a positive 
constant $K_1$ such that 
\begin{align}\label{Y<}
    |Y(s)Y(\zeta)^{-1}|
    <
    K_1|Y(s')Y(\zeta)^{-1}|
    \quad s\in S_c(\theta;\varepsilon).
\end{align}
Since $\arg(s)$ is bounded on 
$S_c(\theta;\varepsilon)$
and the matrix
$N$ in the sum $G=\gamma\id+N$
is a Jordan nilpotent matrix, 
there is a 
polynomial $P$ with positive coefficients
such that 
\begin{align}\label{G<}
    |(s'/\zeta)^G|\leq 
    |s'/\zeta|^{\mathrm{Re}(\gamma)}
    |P(\log |s'/\zeta|)|,
\end{align}
where $|(s'/\zeta)^G|$ denotes an operator norm. 
We also have 
a positive constant $K_2$,
which is independent of 
$\zeta\in C(s)$
such that 
\begin{align}\label{E}
    |1-e^{2\pi\i (s-\zeta)}|
    \leq
    K_2|e^{-2\pi\i (s'-\zeta)}|
\end{align}
for $s\in S_c(\theta;\varepsilon)$ 
and 
$\zeta\in C(s)$. 
\subsubsection{}
We shall divide 
the path $C(s)$ into two terms
as follows:
\begin{align*}
    \ell_0(s)&=
    \{(p+T(s'-p)\mid  0<T\leq 1\},\\
    \ell_1(s)&=
    \{(p+T(s'-p)\mid   T\geq 1\}.
\end{align*}
We then set 
\begin{align*}
    I_j(s)\coloneqq 
    Y(s)
    \int_{{{\ell}_j} (s)}
    \frac{Y(\zeta)^{-1}f(\zeta)}
    {1-e^{ 2\pi\i (s-\zeta) }}d\zeta
\end{align*}
for $j=0,1$.
It is not difficult to see that 
$I_0(s)$ is of moderate growth
(resp. rapidly decay)
when $f$ has the same property. 
We shall give an estimate on $I_1(s)$. 
Set $\bm{\psi}(s)\coloneqq \frak{a}(s)-c_m s$,
$\bm{\varphi}(\zeta)
\coloneqq \bm{\mu} \zeta+\bm{\nu}
\log \zeta+\bm{\psi}(\zeta)$
where we put $\bm{\mu}=c_m-2\pi\i$,
and $\bm{\nu}=-N+\mathrm{Re}(\gamma)$.
Then, combining 
\eqref{Y<}, \eqref{G<},
and \eqref{E}, and 
the assumption \eqref{Moderate},
there exist positive constants
$K_3$ and $N$ such that 
\begin{align*}
    |I_1(s)|&\leq 
     K_3|s'|^N
    \int_{\ell_1(s)} |
    \exp(\bm{\varphi}(s')-\bm{\varphi}
    (\zeta))P(\log|s'/\zeta|)
    d\zeta| 
\end{align*}
Then we can apply 
Lemma \ref{Immink-Estimate} 
to the integral 
by the inequality \eqref{mu}. \qed
\bibliographystyle{alpha}
\bibliography{Difference}
\end{document}